\documentclass{amsproc}

\usepackage{amsmath,amssymb,amscd,amsthm,amsfonts}
\usepackage{latexsym}
\usepackage{stmaryrd}
\usepackage{tikz-cd}
\usepackage[cal=boondox,calscaled=.96]{mathalfa}
\usepackage[debug]{hyperref}

\newcommand{\QQ}{\mathbb{Q}}
\newcommand{\ZZ}{\mathbb{Z}}
\newcommand{\N}{\mathbb{N}}

\newcommand{\sbullet}{{\hspace{.1em}\scriptstyle\bullet\hspace{.1em}}}

\DeclareMathOperator{\ord}{ord}
\DeclareMathOperator{\Aut}{Aut}
\DeclareMathOperator{\supp}{\rm supp}
\DeclareMathOperator{\U}{\mathcal{U}}
\DeclareMathOperator{\Der}{Der}
\DeclareMathOperator{\Hom}{Hom}
\DeclareMathOperator{\End}{End}
\DeclareMathOperator{\HS}{HS}
\DeclareMathOperator{\opvarphi}{\varphi}
\DeclareMathOperator{\Par}{\mathcal{P}}
\DeclareMathOperator{\HH}{\Lambda}

\newcommand{\bchi}{{\boldsymbol\chi}}
\newcommand{\bSigma}{\boldsymbol{\Sigma}}
\newcommand{\bepsilon}{{\boldsymbol\varepsilon}}

\newcommand{\pcirc}{{\scriptstyle \,\circ\,}}
\newcommand{\smallcirc}{{\scriptstyle \circ}}
\newcommand\Id{{\rm Id}}

\newcommand{\bfu}{{\bf u}}
\newcommand{\bfs}{{\bf s}}
\newcommand{\bft}{{\bf t}}

\newcommand{\fd}{\mathfrak{d}}

\newtheorem{thm}{Theorem}[subsection]
\newtheorem{cor}[thm]{Corollary}
\newtheorem{prop}[thm]{Proposition}
\newtheorem{lemma}[thm]{Lemma}

\theoremstyle{definition}
\newtheorem{defi}[thm]{Definition}
\newtheorem{exam}[thm]{Example}
\newtheorem{notacion}[thm]{Notation}

\theoremstyle{remark}
\newtheorem{nota}[thm]{Remark}
\newcommand{\numero}{\refstepcounter{thm}\noindent {\bf  \thethm\ }}

\numberwithin{equation}{section}

\begin{document}

\title{Hasse--Schmidt derivations versus classical derivations}

\author[L. Narv\'{a}ez Macarro]{Luis Narv\'{a}ez Macarro}

\address{ \href{http://personal.us.es/narvaez/}{Luis Narv\'aez Macarro}\\
\noindent \href{http://departamento.us.es/da/}{Departamento de \'Algebra} \&\
\href{http://www.imus.us.es}{Instituto de Matem\'aticas (IMUS)}\\
\href{http://matematicas.us.es}{Facultad de Matem\'aticas}, \href{http://www.us.es}{Universidad de Sevilla}\\
Calle Tarfia s/n, 41012  Sevilla, Spain.}
\email{narvaez@us.es}
\thanks{Partially supported by MTM2016-75027-P, P12-FQM-2696
 and FEDER}

\subjclass[2010]{Primary: 14F10, 13N10; Secondary 13N15}

\date{}

\dedicatory{Dedicated to L\^e D\~ung Tr\'ang}

\keywords{Hasse--Schmidt derivation, derivation, power series ring, substitution map}

\begin{abstract}
In this paper we survey the notion and basic results on multivariate Hasse--Schmidt derivations 
over arbitrary commutative algebras and we associate to such an object a family of classical derivations. We study the behavior of these derivations under the action of substitution maps and we prove that, in characteristic $0$, the original multivariate Hasse--Schmidt derivation can be recovered from the associated family of classical derivations. Our constructions generalize a previous one by M.~Mirzavaziri in the case of a base field of characteristic $0$.
\end{abstract}

\maketitle

\section*{Introduction}

Let $k$ be a commutative ring and $A$ a commutative $k$-algebra. A Hasse--Schmidt derivation of $A$ over $k$ of length $m\geq 0$ (or $m=\infty$), is a sequence $D=(D_0,D_1,\dots,D_m)$ (or $D=(D_0,D_1,\dots)$) of $k$-linear endomorphisms of $A$ such that $D_0$ is the identity map and 
$$ D_\alpha (x y) = \sum_{\scriptscriptstyle \beta + \gamma = \alpha} D_\beta(x) D_\gamma(y),\quad \forall \alpha, \forall x,y\in A.
$$
A such $D$ can be seen as a power series $D=\sum_{\alpha=0}^m D_\alpha t^\alpha$ in the quotient ring $R[[t]]/\langle t^{m+1} \rangle$, with $R=\End_k(A)$ (the ring of endomorphisms of $A$ as $k$-module).
For $i\geq 1$, the $i$th component $D_i$ turns out to be a $k$-linear differential operator of order $\leq i$ vanishing on $1$, in particular $D_1$ is a $k$-derivation of $A$.
\medskip

The notion of Hasse--Schmidt derivation was introduced in \cite{has37} in the case where $k$ is a field of characteristic $p>0$ and $A$ a field of algebraic functions over $k$. This notion was used to understand, among others, Taylor expansions in this setting. But actually, Hasse--Schmidt derivations make sense in full generality.
\medskip

If we are in characteristic $0$ ($\QQ \subset k$), then it is easy to produce examples of Hasse--Schmidt derivations: starting with a $k$-linear derivation $\delta:A\to A$ we consider its exponential:
$$ e^{t\delta} = \sum_{\alpha=0}^\infty \frac{\delta^\alpha}{\alpha!} t^\alpha \in R[[t]],\quad R=\End_k(A).
$$
It is clear that $ e^{t\delta}$ is a Hasse--Schmidt derivation of $A$ over $k$ (of infinite length). This example also proves that, always under the characteristic $0$ hypothesis, any $k$-linear derivation $\delta:A\to A$  appears as the 1-component of some Hasse--Schmidt derivation (of infinite length, and so, of any length $m\geq 1$) of $A$ over $k$. This is what we call ``to be $\infty$-integrable'' (and so ``$m$-integrable'', for each $m\geq 1$) (see \cite{mat-intder-I}). But if we are no more in characteristic $0$, the situation becomes much more involved and integrable derivations deserve special attention (see \cite{nar_2012,hoff_kow_2015,MPTirado_2018} for several recent achievements in that direction).
\medskip

As far as the author knows, two papers have been concerned with the description of Hasse--Schmidt derivations in terms of usual derivations, both in the case where $k$ is a field of characteristic $0$.
In \cite{heer70} it is proven that\footnote{This result has been ``rediscovered'' in \cite{Saymeh_1986} for $A$ a commutative algebra over a field $k$ of characteristic zero.}, if $A$ is a (possibly non-commutative) $k$-algebra, then any Hasse--Schmidt derivation $D=(D_0=\Id,D_1,D_2,\dots)$ of infinite length of $A$ over $k$ is determined by a unique sequence $\delta=(\delta_1,\delta_2,\dots )$ of classical derivations $\delta_i\in\Der_k(A)$. Namely, the expressions relating $D$ and $\delta$ are:
$$
\delta_n = \sum_{r=1}^n \frac{(-1)^{r+1}}{r} \sum_{\substack{\scriptscriptstyle n_1 + \cdots + n_r=n\\ \scriptscriptstyle n_i>0 }} D_{n_1} D_{n_2} \cdots D_{n_r},\   D_n = \sum_{r=1}^n \frac{1}{r!} \sum_{\substack{\scriptscriptstyle n_1 + \cdots + n_r=n\\ \scriptscriptstyle n_i>0 }} \delta_{n_1} \delta_{n_2} \cdots \delta_{n_r},
$$
or in other words:
$$ \sum_{n=0}^\infty D_n t^n = \exp \left( \sum_{n=1}^\infty \delta_n t^n\right).
$$
A similar result is proven in \cite{Mirza_2010}: any Hasse--Schmidt derivation $D=(D_0=\Id,D_1,D_2,\dots)$ of infinite length of $A$ over $k$ determines, and is determined by a sequence $\delta = (\delta_1,\delta_2,\dots)$ of classical derivations given by the following recursive formula:
$$ (n+1) D_{n+1} = \sum_{r=0}^n \delta_{r+1} D_{n-r},\quad n\geq 0.
$$
An interesting reinterpretation of both results can be found in \cite{haze_2011}.
\medskip

The goal of this paper is twofold: to give a survey of multivariate Hasse--Schmidt derivations over a general commutative base ring $k$ and a general commutative $k$-algebra $A$, as defined in \cite{nar_subst_2018}; and to generalize the construction in \cite{Mirza_2010} to this setting.
\medskip

One of our motivations is to understand the relationship between HS-modules, as defined in \cite{nar_envelop}, and classical integrable connections. The paper \cite{nar_HSmod_vs_IC} 
is devoted to prove that both notions are equivalent in characteristic $0$, and the proof strongly depends on the constructions and results of the present paper.
\medskip

A $(p,\Delta)$-variate Hasse--Schmidt derivation of $A$ over $k$ is a family  $D= \left(D_\alpha\right)_{\alpha\in\Delta}$ of $k$-linear endomorphisms of $A$ such that $D_0$ is the identity map and:
$$ D_\alpha (x y) = \sum_{\scriptscriptstyle \beta + \gamma = \alpha} D_\beta(x) D_\gamma(y),\quad \forall \alpha\in \Delta, \forall x,y\in A,
$$
where $\Delta \subset \N^p$ is a non-empty {\em co-ideal}, i.e. a subset of $\N^p$ such that everytime $\alpha\in \Delta$ and $\alpha'\leq \alpha$ (i.e. $\alpha - \alpha' \in \N^p$)  we have $\alpha'\in \Delta$. A simple but important idea is to think on Hasse--Schmidt derivations as series $D=\sum_{\scriptscriptstyle \alpha\in\Delta} D_\alpha \bfs^\alpha$ in the quotient ring $R[[\bfs]]_\Delta$ of the power series ring $R[[\bfs]] = R[[s_1,\dots,s_p]]$, $R=\End_k(A)$, by the two-sided monomial ideal generated by all $\bfs^\alpha$ with $\alpha \in \N^p \setminus \Delta$. 
\medskip

The set $\HS^p_k(A;\Delta)$ of $(p,\Delta)$-variate Hasse--Schmidt derivations is a subgroup of the group of units $\left(R[[\bfs]]_\Delta\right)^\times$, and it also carries the action of substitution maps: 
 given a substitution map $\varphi: A[[s_1,\dots,s_p]]_\Delta \to A[[t_1,\dots,t_q]]_\nabla $ and a $(p,\Delta)$-variate Hasse--Schmidt derivation $D=\sum_{\scriptscriptstyle \alpha\in\Delta} D_\alpha \bfs^\alpha$ we obtain a new ($q,\nabla)$-variate Hasse--Schmidt derivation given by:
$$\varphi\sbullet D := \sum_{\scriptscriptstyle \alpha\in\Delta} \varphi(\bfs^\alpha) D_\alpha.
$$
This new structure is a key point in \cite{nar_envelop}.
\medskip

To generalize the construction in \cite{Mirza_2010}, we reinterpret the aforementioned recursive formula by means of the ``logarithmic derivative type'' maps:
$$  \varepsilon^i:  D \in \HS^p_k(A;\Delta) \longmapsto \varepsilon^i(D) := D^* \left( s_i \frac{\partial D}{\partial s_i}\right) \in R[[\bfs]]_\Delta,\quad i=1,\dots,p,
$$
where $D^*$ denotes the inverse of $D$. The starting point is to check that the coefficients of $\varepsilon(D)$ are always classical derivations, i.e. $\varepsilon(D) \in \Der_k(A)[[\bfs]]_\Delta$.
\bigskip

Let us comment on the content of the paper.
\medskip

In section 1 we have gathered some notations and constructions on power series modules, powers series rings and substitution maps, most of them taken from \cite{nar_subst_2018}, and we study the maps $\varepsilon^i$,  and their conjugate $\overline{\varepsilon}^i$.
\medskip

In section 2 we recall the notion and the basic properties of multivariate Hasse--Schmidt derivations and of the action of substitution maps on these objects.
\medskip

Section 3 contains the main original results of this paper. First, we see how the $\varepsilon^i$ or $\overline{\varepsilon}^i$ maps of section 1 allow us to associate to any multivariate Hasse--Schmidt derivation a power series whose coefficients are classical derivations, as explained before. When $\QQ\subset k$ we obtain a characterization of multivariate Hasse--Schmidt derivations in terms of the $\varepsilon^i$ (or $\overline{\varepsilon}^i$) maps, and we prove that any multivariate Hasse--Schmidt derivation can be constructed from a power series of classical derivations. To finish, we study the behavior of the $\varepsilon^i$ maps of a multivariate Hasse--Schmidt derivation under the action of substitution maps.

\section{Notations and preliminaries}

\subsection{Notations}

Throughout the paper we will use the following notations:
\smallskip

\noindent -) $k$ is a commutative ring and $A$ a commutative
$k$-algebra.
\smallskip

\noindent -) $\bfs = \{s_1,\dots,s_p\}$, $\bft =\{t_1,\dots,t_q\}$, \dots\  are sets of variables.
\smallskip

\noindent -) $\U^p(R;\Delta)$: 
see Notation \ref{notacion:Ump}.
\smallskip

\noindent -) ${\bf C}_e(\varphi,\alpha)$: see (\ref{eq:exp-substi}).
\smallskip

\noindent -) $\varphi \sbullet r, r\sbullet \varphi$: see \ref{nume:def-sbullet}.
\smallskip

\noindent -) $\varphi^D$: see \ref{prop:varphi-D-main}.
\smallskip

\subsection{Some constructions on power series rings and modules}

Throughout this section, $k$ will be a commutative ring, $A$ a commutative $k$-algebra and $R$ a ring, not-necessarily commutative.
\medskip

Let $p\geq 0$ be an integer and let us call $\bfs = \{s_1,\dots,s_p\}$ a set of $p$ variables. The support of each $\alpha\in \N^p$ is defined as  $\supp \alpha := \{i\ |\ \alpha_i\neq 0\}$. 
The monoid $\N^p$ is endowed with a natural partial ordering. Namely, for $\alpha,\beta\in \N^p$, we define:
$$ \alpha \leq \beta\quad \stackrel{\text{def.}}{\Longleftrightarrow}\quad \exists \gamma \in \N^p\ \text{such that}\ \beta = \alpha + \gamma\quad \Longleftrightarrow\quad \alpha_i \leq \beta_i\quad \forall i=1\dots,p.$$
We denote $|\alpha| := \alpha_1+\cdots+\alpha_p$. 
\medskip

Let $p\geq 1$ be an integer and $\bfs = \{s_1,\dots,s_p\}$ a set of variables. 
If $M$ is an abelian group and $M[[\bfs]]$ is the abelian group of power series with coefficients in $M$,
the {\em support}  of a series $m=\sum_\alpha m_\alpha \bfs^\alpha \in M[[\bfs]]$ is $\supp(m) := \{  \alpha \in \N^p\ |\  m_\alpha \neq 0\} \subset \N^p$. We have $m=0\Leftrightarrow \supp(m) = \emptyset$. 
\medskip

The abelian group $M[[\bfs]]$ is clearly a $\ZZ[[\bfs]]$-module, which will be always endowed with the $\langle \bfs \rangle$-adic topology.

\begin{defi} We say that a subset $\Delta \subset \N^p$ is an {\em ideal} (resp. a
{\em co-ideal}) of $\N^p$  
if everytime $\alpha\in\Delta$ and $\alpha\leq \alpha'$ (resp. $\alpha'\leq \alpha$), then $\alpha'\in\Delta$.
\end{defi}
It is clear that $\Delta \subset \N^p$ is an ideal if and only if its complement $\Delta^c$ is a co-ideal, and that the union and the intersection of any family of ideals (resp. of co-ideals)  of $\N^p$ is again an ideal (resp. a co-ideal) of $\N^p$. 
Examples of ideals (resp. of co-ideals) of $\N^p$ are the $\beta + \N^p$
(resp. the $\{\alpha \in \N^p\ |\ \alpha\leq \beta\}$ ) with $\beta \in \N^p$. The  $\{\alpha \in \N^p\ |\ |\alpha|\leq m\}$ with $m\geq 0$ are also co-ideals.
Notice that a co-ideal $\Delta\subset \N^p$ is non-empty if and only if $\{0\}  \subset \Delta$.
\medskip

\numero \label{nume:M[[bfs]]Delta}
Let $M$ be an abelian group. 
For each co-ideal $\Delta \subset \N^p$, we denote by $\Delta_M$ the closed sub-$\ZZ[[\bfs]$-bimodule of $M[[\bfs]]$ whose elements are the formal power series $\sum_{\alpha\in\N^p} m_\alpha \bfs^\alpha$ such that $m_\alpha=0$ whenever $\alpha\in \Delta$, and $M[[\bfs]]_\Delta := M[[\bfs]]/\Delta_M$. The elements in $M[[\bfs]]_\Delta$ are power series of the form 
$\sum_{\scriptscriptstyle \alpha\in\Delta} m_\alpha \bfs^\alpha$, $m_\alpha \in M$. If $f:M\to M'$ is a homomorphism of abelian groups, we will denote by $\overline{f}: M[[\bfs]]_\Delta \to M'[[\bfs]]_\Delta$ the $\ZZ[[\bfs]]_\Delta$-linear map defined as $\overline{f}\left(\sum_{\scriptscriptstyle \alpha\in\Delta} m_\alpha \bfs^\alpha\right) = \sum_{\scriptscriptstyle \alpha\in\Delta} f(m_\alpha) \bfs^\alpha$.
\medskip

If $R$ is a ring, then $\Delta_R$ is a closed two-sided ideal of $R[[\bfs]]$ and so $R[[\bfs]]_\Delta$ is a topological ring, which we always consider endowed with the $\langle \bfs \rangle$-adic topology ( $=$ to the quotient topology). Similarly, if $M$ is an $(A;A)$-bimodule (central over $k$), then $M[[\bfs]]_\Delta$ is an $(A[[\bfs]_\Delta;A[[\bfs]]_\Delta)$-bimodule (central over $k[[\bfs]]_\Delta$).
\medskip

For $\Delta' \subset \Delta$ non-empty co-ideals of $\N^p$, we have natural $\ZZ[[\bfs]]$-linear projections 
$\tau_{\Delta \Delta'}:M[[\bfs]]_{\Delta} \longrightarrow M[[\bfs]]_{\Delta'}$,that we call {\em truncations}:
$$\tau_{\Delta \Delta'} : \sum_{\scriptscriptstyle \alpha\in\Delta} m_\alpha \bfs^\alpha \in M[[\bfs]]_{\Delta'} \longmapsto \sum_{\scriptscriptstyle \alpha\in\Delta'} m_\alpha \bfs^\alpha \in M[[\bfs]]_{\Delta}.
$$
If $M$ is a ring (resp. an $(A;A)$-bimodule), then the truncations $\tau_{\Delta \Delta'}$ are ring homomorphisms (resp.  $(A[[\bfs]]_{\Delta};A[[\bfs]]_{\Delta})$-linear maps). For $\Delta' = \{0\}$ we have $ M[[\bfs]]_{\Delta'} = M$ and the kernel of $\tau_{\Delta \{0\}}$ will be denoted by $M[[\bfs]]_{\Delta,+}$. 
We have a bicontinuous isomorphism:
$$ M[[\bfs]]_\Delta = \lim_{\longleftarrow} M[[\bfs]]_{\Delta'}
$$
where $\Delta'$ runs over all
finite co-ideals contained in $\Delta$.

\begin{defi} \label{def:k-algebra-over-A}
A $k$-algebra over $A$  is a (not-necessarily commutative) $k$-algebra $R$ endowed with a map of $k$-algebras $\iota:A \to R$. A map between two $k$-algebras $\iota:A \to R$ and $\iota':A \to R'$ over $A$ is a map $g:R\to R'$ of $k$-algebras such that $\iota' = g \pcirc \iota$. 
\end{defi}

It is clear that if $R$ is a $k$-algebra over $A$, then $R[[\bfs]]_\Delta$ is a $k[[\bfs]]_\Delta$-algebra over $A[[\bfs]]_\Delta$.

\begin{notacion} \label{notacion:Ump} Let $R$ be a ring, $p\geq 1$ and  $\Delta\subset \N^p$ a non-empty co-ideal. We denote by $\U^p(R;\Delta)$ the multiplicative sub-group of the units of $R[[\bfs]]_\Delta$ whose 0-degree coefficient is $1$. The multiplicative inverse of a unit $r\in R[[\bfs]]_\Delta$ will be denoted by $r^*$.
Clearly, $\U^p(R;\Delta)^{\text{\rm opp}} = \U^p(R^{\text{\rm opp}};\Delta)$.
For $\Delta\subset \Delta'$ co-ideals we have $\tau_{\Delta'\Delta}\left(\U^p(R;\Delta')\right) \subset \U^p(R;\Delta)$ and the truncation map
$\tau_{\Delta'\Delta}:\U^p(R;\Delta') \to \U^p(R;\Delta)$ is a group homomorphisms. Clearly, we have:
\begin{equation} \label{eq:U-inv-limit-finite}
   \U^p(R;\Delta) =  \lim_{\stackrel{\longleftarrow}{\substack{\scriptscriptstyle \Delta' \subset \Delta\\ \scriptscriptstyle  \sharp \Delta'<\infty}}} \U^p(R;\Delta').
\end{equation}
If $p=1$ and $\Delta =  \{i\in\N\ |\ i\leq m\}$ we will simply denote $\U(R;m):= \U^1(R;\Delta)$.
\medskip

For any ring homomorphism $f:R\to R'$, the induced ring homomorphism $\overline{f}: R[[\bfs]]_\Delta \to R'[[\bfs]]_\Delta$ sends $\U^p(R;\Delta)$ into $\U^p(R';\Delta)$ and so it induces natural group homomorphisms
$\U^p(R;\Delta) \to \U^p(R';\Delta)$. 
\end{notacion}

We recall the following easy result (cf. Lemma 2 in \cite{nar_subst_2018}).

\begin{lemma} \label{lemma:unit-psr}
Let $R$ be a ring and $\Delta\subset \N^p$ a non-empty co-ideal. 
The units in $R[[\bfs]]_\Delta$ are those power series $r=\sum r_\alpha \bfs^\alpha$ such that $r_0$ is a unit in $R$. Moreover, in the special case where $r_0=1$, the inverse $r^*= \sum r^*_\alpha \bfs^\alpha$ of $r$ is given by
$r^*_0 = 1$ and
$$ r^*_\alpha =  \sum_{d=1}^{|\alpha|} (-1)^d
\sum_{\alpha^\bullet \in \Par(\alpha,d)}
r_{\alpha^1} \cdots  r_{\alpha^d}\quad \text{for\ }\ \alpha\neq 0,
$$
where $\Par(\alpha,d)$ is the set of $d$-uples $\alpha^\bullet=(\alpha^1,\dots,\alpha^d)$ with $\alpha^i \in \N^{(\bfs)}$, $\alpha^i\neq 0$, and $\alpha^1+\cdots + \alpha^d=\alpha$.
\end{lemma}
\bigskip

\numero \label{nume:widetilde-r}
 Let $E,F$ be $A$-modules. For each $r= \sum_\beta r_\beta \bfs^\beta \in \Hom_k(E,F)[[\bfs]]_\Delta$ we denote by $\widetilde{r}: E[[\bfs]]_\Delta \to F[[\bfs]]_\Delta$ the map defined by:
$$ \widetilde{r} \left( \sum_{\scriptscriptstyle \alpha\in \Delta} e_\alpha \bfs^\alpha \right) :=
\sum_{\scriptscriptstyle \alpha\in \Delta } \left( \sum_{\scriptscriptstyle \beta + \gamma=\alpha} r_\beta(e_\gamma) \right) \bfs^\alpha,
$$
which is obviously a $k[[\bfs]]_\Delta$-linear map. 
It is clear that the map:
\begin{equation} \label{eq:tilde-map}
 r\in \Hom_k(E,F)[[\bfs]]_\Delta \longmapsto \widetilde{r} \in \Hom_{k[[\bfs]]_\Delta}(E[[\bfs]]_\Delta,F[[\bfs]]_\Delta)
\end{equation}
is $(A[[\bfs]]_\Delta;A[[\bfs]]_\Delta)$-linear.
\medskip

If $f:E[[\bfs]]_\Delta \to F[[\bfs]]_\Delta$ is a $k[[\bfs]]_\Delta$-linear map, let us denote by $f_\alpha:E \to F$, $\alpha\in\Delta $, the $k$-linear maps defined by:
$$ f(e) = \sum_{\scriptscriptstyle \alpha\in \Delta} f_\alpha(e) \bfs^\alpha,\quad \forall e\in E.
$$
If $g:E\to F[[\bfs]]_\Delta$ is a $k$-linear map, we denote by $g^e:E[[\bfs]]_\Delta \to F[[\bfs]]_\Delta$ the unique $k[[\bfs]]_\Delta$-linear map extending $g$ to $E[[\bfs]]_\Delta = k[[\bfs]]_\Delta \widehat{\otimes}_k E$. It is given by:
\begin{equation}  \label{eq:g^e}
 g^e \left( \sum_\alpha e_\alpha \bfs^\alpha \right) := \sum_\alpha g(e_\alpha) \bfs^\alpha.
\end{equation}
We have a $k[[\bfs]]_\Delta$-bilinear and $A[[\bfs]]_\Delta$-balanced map:
$$ \langle -,-\rangle : (r,e) \in \Hom_k(E,F)[[\bfs]]_\Delta \times E[[\bfs]]_\Delta \longmapsto \langle r,e\rangle := \widetilde{r}(e) \in F[[\bfs]]_\Delta.
$$
The following assertions are clear (see \cite[Lemma 3]{nar_subst_2018}):
\begin{enumerate}
\item[1)] The map (\ref{eq:tilde-map})
 is an isomorphism of $(A[[\bfs]]_\Delta;A[[\bfs]]_\Delta)$-bimodules. When $E=F$ it is an isomorphism of 
$k[[\bfs]]_\Delta$-algebras over $A[[\bfs]]_\Delta$.
\item[2)] The restriction map:
$$f \in  \Hom_{k[[\bfs]]_\Delta}(E[[\bfs]]_\Delta,F[[\bfs]]_\Delta) \mapsto f|_E \in \Hom_k(E,F[[\bfs]]_\Delta),
$$ 
is an isomorphism of $(A[[\bfs]]_\Delta;A)$-bimodules.
\end{enumerate}

Let us call $R = \End_k(E)$. As a consequence of the above properties, the composition of the maps:
\begin{equation}  \label{eq:comple_formal_1}
 R[[\bfs]]_\Delta \xrightarrow{r \mapsto \widetilde{r}} \End_{k[[\bfs]]_\Delta}(E[[\bfs]]_\Delta) \xrightarrow{f \mapsto f|_E} \Hom_k(E,E[[\bfs]]_\Delta)
\end{equation}
is an isomorphism of $(A[[\bfs]]_\Delta;A)$-bimodules, and so $\Hom_k(E,E[[\bfs]]_\Delta)$ inherits a natural structure of $k[[\bfs]]_\Delta$-algebra over $A[[\bfs]]_\Delta$. Namely, if $g,h:E \to E[[\bfs]]_\Delta$ are $k$-linear maps
with:
$$ g(e)=\sum_{\scriptscriptstyle \alpha\in \Delta} g_\alpha(e)\bfs^\alpha,\ h(e)=\sum_{\scriptscriptstyle \alpha\in \Delta} h_\alpha(e)\bfs^\alpha,\quad \forall e\in E,\quad g_\alpha, h_\alpha \in \Hom_k(E,E),
$$
then the product $h g \in \Hom_k(E,E[[\bfs]]_\Delta)$ is given by:
\begin{equation} \label{eq:product-HomEE[[s]]}
 (hg)(e) = \sum_{\scriptscriptstyle \alpha\in \Delta} \left( \sum_{\scriptscriptstyle \beta + \gamma = \alpha} (h_\beta \pcirc g_\gamma)(e) \right) \bfs^\alpha.
\end{equation}

\begin{notacion} \label{notacion:pcirc}
We denote:
$$
\Hom_k^\pcirc(E,E[[\bfs]]_\Delta) := \left\{ f \in \Hom_k(E,E[[\bfs]]_\Delta)\ |\ f(e) \equiv e\!\!\!\!\mod \langle \bfs \rangle E[[\bfs]]_\Delta\ \forall e\in E \right\}, 
$$
\begin{eqnarray*}
&\Aut_{k[[\bfs]]_\Delta}^\pcirc(E[[\bfs]]_\Delta) 
:=
&
\\
& \left\{    
f \in \Aut_{k[[\bfs]]_\Delta}(E[[\bfs]]_\Delta)\ |\   f(e) \equiv e_0\!\!\!\!\mod \langle \bfs \rangle E[[\bfs]]_\Delta\ \forall e\in E[[\bfs]]_\Delta  \right\}.&
\end{eqnarray*}
Let us notice that a $f \in \Hom_k(E,E[[\bfs]]_\Delta)$, given by $f(e) = \sum_{\alpha\in \Delta} f_\alpha(e) \bfs^\alpha$, belongs to $\Hom_k^\pcirc(E,E[[\bfs]]_\Delta)$ if and only if $f_0=\Id_E$.
\end{notacion}

The isomorphism in (\ref{eq:comple_formal_1}) gives rise to a group isomorphism:
\begin{equation} \label{eq:U-iso-pcirc}
 r\in \U^p(\End_k(E);\Delta) \stackrel{\sim}{\longmapsto} \widetilde{r} \in \Aut_{k[[\bfs]]_\Delta}^\pcirc(E[[\bfs]]_\Delta) 
\end{equation}
and to a bijection:
\begin{equation} \label{eq:Aut-iso-pcirc} 
f\in \Aut_{k[[\bfs]]_\Delta}^\pcirc(E[[\bfs]]_\Delta) \stackrel{\sim}{\longmapsto} f|_E \in \Hom_k^\pcirc(E,E[[\bfs]]_\Delta).
\end{equation}
So, $\Hom_k^\pcirc(E,E[[\bfs]]_\Delta)$ is naturally a group with the product described in (\ref{eq:product-HomEE[[s]]}).
\bigskip

If $R$ is a (not necessarily commutative) $k$-algebra and $\Delta \subset \N^p$ is a co-ideal, any continuous $k$-linear map $h: k[[\bfs]]_\Delta \to k[[\bfs]]_\Delta$ induces a natural continuous left and right $R$-linear map:
$$ h_R:= \Id_R \widehat{\otimes}_k h:R[[\bfs]]_\Delta = R \widehat{\otimes}_k k[[\bfs]]_\Delta \longrightarrow R[[\bfs]]_\Delta = R \widehat{\otimes}_k k[[\bfs]]_\Delta 
$$
given by:
$$ h_R\left(\sum_\alpha r_\alpha \bfs^\alpha \right) = \sum_\alpha r_\alpha h(\bfs^\alpha).
$$
If $\fd:k[[\bfs]]_\Delta \to k[[\bfs]]_\Delta$ is $k$-derivation, it is continuous and $\fd_R: R[[\bfs]]_\Delta \to R[[\bfs]]_\Delta$ is a  $(R;R)$-linear derivation, i.e. 
$ \fd_R(sr)=s\fd_R(r)$, $\fd_R(rs)=\fd_R(r)s$, $\fd_R( r r') =  \fd_R ( r ) r' + r \fd_R( r')$ for all $s\in R$ and for all $r,r'\in R[[\bfs]]_\Delta$.
\medskip

The set of all $(R;R)$-linear derivations of $R[[\bfs]]_\Delta$ is a $k[[\bfs]]_\Delta$-Lie algebra and will be denoted by $\Der_R(R[[\bfs]]_\Delta)$. Moreover,
the map:
$$ \fd \in \Der_k(k[[\bfs]]_\Delta) \longmapsto \fd_R \in \Der_R(R[[\bfs]]_\Delta)
$$
is clearly a map of $k[[\bfs]]_\Delta$-Lie algebras.
\medskip

The following definition provides a particular family of $k$-derivations.

\begin{defi} \label{def:Euler-derivation}
For each $i=1,\dots,p$, 
the {\em $i$th partial Euler $k$-derivation}  is $\chi^i = s_i\frac{\partial}{\partial s_i}:k[[\bfs]] \to k[[\bfs]]$. It induces a $k$-derivation on each $k[[\bfs]]_\Delta$, which will be also denoted by $\chi^i$.
\medskip

\noindent 
The {\em Euler $k$-derivation} $\bchi:k[[\bfs]] \to k[[\bfs]]$
is defined as:
$$ \bchi = \sum_{i=1}^p \chi^i,\quad
\bchi\left( \sum_{\scriptscriptstyle \alpha} c_\alpha \bfs^\alpha \right) =  \sum_{\scriptscriptstyle \alpha} |\alpha| c_\alpha \bfs^\alpha.
$$
It induces a $k$-derivation on each $k[[\bfs]]_\Delta$, which will be also denoted by $\chi$.
\end{defi}

The proof of the following lemma is easy and it is left to the reader.

\begin{lemma} \label{lemma:carac-der-bfs}
Let $E$ be an $A$-module and
 $r= \sum_\beta r_\beta \bfs^\beta \in \Hom_k(A,E)[[\bfs]]_\Delta$ a formal power series with coefficients in $\Hom_k(A,E)$. The following properties are equivalent:
\begin{enumerate}
\item[(1)] $r \in \Der_k(A,E)[[\bfs]]_\Delta$.
\item[(2)] For any $a\in A[[\bfs]]_\Delta$ we have $[r,a] = \widetilde{r}(a)$.
\item[(3)] $\widetilde{r} \in \Der_{k[[\bfs]]_\Delta}(A[[\bfs]]_\Delta,E[[\bfs]]_\Delta)$.
\item[(4)] $\widetilde{r}|_A \in \Der_k(A,E[[\bfs]]_\Delta)$.
\end{enumerate}
\end{lemma}

In particular, for each $r\in \Der_k(A)[[\bfs]]_\Delta $, we have that $\widetilde{r} \in  
\Der_{k[[\bfs]]_\Delta}(A[[\bfs]]_\Delta)$ (see \ref{nume:widetilde-r}) and that the $A[[\bfs]]_\Delta$-linear map
\begin{equation} \label{eq:Der-s-anchor}
 r\in \Der_k(A)[[\bfs]]_\Delta \longmapsto \widetilde{r} \in \Der_{k[[\bfs]]_\Delta}(A[[\bfs]]_\Delta)
\end{equation}
is an isomorphism of $A[[\bfs]]_\Delta$-modules. Moreover, 
$\Der_k(A)[[\bfs]]_\Delta$ is a Lie algebra over $k[[\bfs]]_\Delta$, where the Lie bracket of $\delta = \sum_\alpha \delta_\alpha \bfs^\alpha, \varepsilon = \sum_\alpha \varepsilon_\alpha \bfs^\alpha \in 
\Der_k(A)[[\bfs]]_\Delta$ is given by:
$$ [\delta,\varepsilon] = \delta\varepsilon - \varepsilon \delta = \sum_\alpha \left( \sum_{\scriptscriptstyle \beta + \gamma =\alpha} [\delta_\beta,\varepsilon_\gamma] \right) \bfs^\alpha,
$$
and the map (\ref{eq:Der-s-anchor}) is also an isomorphism of $k[[\bfs]]_\Delta$-Lie algebras. 

\begin{lemma} \label{lemma:leibniz-<>}
Let $\fd:k[[\bfs]]_\Delta \to k[[\bfs]]_\Delta$ be a $k$-derivation and $R=\End_k(A)$. Then, for each $r\in R[[\bfs]]_\Delta$ we have
$\widetilde{\fd_R(r)} = [\fd_A,\widetilde{r}]$.
\end{lemma}

\begin{proof} 
We have to prove that $ \fd_A \left( \langle r, a \rangle \right) = \langle \fd_R (r), a \rangle +  \langle r, \fd_A(a) \rangle$
for all $a\in A[[\bfs]]_\Delta$. 
By continuity, it is enough to prove the identity for $r=r_\alpha \bfs^\alpha$, $a=a_\beta \bfs^\beta$ with $\alpha,\beta\in\Delta$, $r_\alpha \in R$, $a_\beta\in A$:
\begin{eqnarray*}
&
\fd_A \left( \langle r, a \rangle \right) = \fd_A \left( \widetilde{r}(a) \right) =
\fd_A(r_\alpha(a_\beta) \bfs^\alpha \bfs^\beta) = 
r_\alpha(a_\beta) \fd (\bfs^\alpha) \bfs^\beta + r_\alpha(a_\beta) \bfs^\alpha \fd (\bfs^\beta )=&
\\
&
\widetilde{\fd_R(r)} (a) + \widetilde{r} (\fd_A(a))=\langle \fd_R (r), a \rangle +
 \langle r, \fd_A(a) \rangle.
\end{eqnarray*}
\end{proof}

\begin{defi} \label{defi:varepsilon-fd}
For any $k$-derivation $\fd:k[[\bfs]]_\Delta \to k[[\bfs]]_\Delta$ and any $r\in \U^p(R;\Delta)$ we define:
$$\varepsilon^\fd(r) :=r^* \fd_R(r),\quad \overline{\varepsilon}^\fd(r) :=\fd_R(r) r^*,
$$
and we will write:
$$ \varepsilon^\fd(r)= \sum_{\alpha} \varepsilon^\fd_\alpha(r) \bfs^\alpha,\quad 
 \overline{\varepsilon}^\fd(r)= \sum_{\alpha} \overline{\varepsilon}^\fd_\alpha(r) \bfs^\alpha.
$$
We will simply denote:
\begin{itemize}
\item[-)] $ \varepsilon^i(r) := \varepsilon^\fd(r)$, $\overline{\varepsilon}^i(r) :=  \overline{\varepsilon}^\fd(r)
$ if $\fd= \chi^i$ (the $i$th partial Euler derivation), $i=1,\dots,p$.
\item[-)]  $\varepsilon(r) := \varepsilon^\fd(r)$, $\overline{\varepsilon}(r) :=  \overline{\varepsilon}^\fd(r)$ if $\fd=\bchi$ is the Euler derivation.
\end{itemize}
\end{defi}

\noindent 
Observe that $\overline{\varepsilon}^\fd(r) = r \, \varepsilon^\fd(r) \, r^*$ and,
for any co-ideal $\Delta'\subset \Delta$, we have $\tau_{\Delta \Delta'} (\varepsilon^\fd(r))=
\varepsilon^\fd(\tau_{\Delta \Delta'} (r))$, $\tau_{\Delta \Delta'} (\overline{\varepsilon}^\fd(r))=
\overline{\varepsilon}^\fd(\tau_{\Delta \Delta'} (r))$. Moreover, if $E$ is an $A$-module and $R=\End_k(A)$, then
$$ \widetilde{\varepsilon^\fd(r)} = \widetilde{r}^{-1} [\fd_A,\widetilde{r}] = \widetilde{r}^{-1} \fd_A\widetilde{r} - \fd_A,\quad \widetilde{\overline{\varepsilon}^\fd(r)} = [\fd_A,\widetilde{r}] \widetilde{r}^{-1} = \fd_A - \widetilde{r} \fd_A\widetilde{r}^{-1}.
$$

The proof of the following lemma is straightforward.

\begin{lemma} For each $r\in \U^p_k(R;\Delta)$, the maps:
\begin{eqnarray*}
&
 \fd \in \Der_k(k[[\bfs]]_\Delta) \longmapsto \varepsilon^\fd(r) \in R[[\bfs]]_\Delta,\ \  
\fd \in \Der_k(k[[\bfs]]_\Delta) \longmapsto \overline{\varepsilon}^\fd(r) \in R[[\bfs]]_\Delta
\end{eqnarray*}
are $k[[\bfs]]_\Delta$-linear. 
\end{lemma}

In particular:
$$ \varepsilon(r) = \sum_{i=1}^p \varepsilon^i(r),\ \ \overline{\varepsilon}(r) = \sum_{i=1}^p \overline{\varepsilon}^i(r).
$$

\begin{lemma} \label{lemma:first-properties-epsilon}
Let $\fd, \fd':k[[\bfs]]_\Delta \to k[[\bfs]]_\Delta$ be 
 $k$-derivations and $r,r'\in\U^p_k(R;\Delta)$. Then, the following identities hold:
\begin{enumerate}
\item[(i)] $\overline{\varepsilon}^\fd(1) = \varepsilon^\fd(1) =0$, $\varepsilon^\fd(r'\, r) = \varepsilon^\fd(r) + r^*\, \varepsilon^\fd(r')\, r$, 
$\overline{\varepsilon}^\fd(r\, r') = \overline{\varepsilon}^\fd(r) + r\, \overline{\varepsilon}^\fd(r')\, r^*$.
\item[(ii)] $\varepsilon^\fd(r^*) = -r\, \varepsilon^\fd(r)\, r^* = -\overline{\varepsilon}^\fd(r)$.
\item[(iii)] $\varepsilon^{[\fd,\fd']}(r) = \left[\varepsilon^\fd(r),\varepsilon^{\fd'}(r)\right] + \fd_R \left( \varepsilon^{\fd'}(r) \right) - \fd'_R \left( \varepsilon^\fd(r) \right) $.
\end{enumerate}
\end{lemma}

\begin{proof}  The proof of (i) is straightforward. For (ii) and (iii) one uses that $\fd_R(r^*) = -r^*\,  \fd_R(r)\, r^*$.
\end{proof}
\medskip

\numero \label{nume:explicit_varepsilon} 
For each $r\in \U^p(R;\Delta)$ and each $i=1,\dots,p$ we have:
\begin{eqnarray*}
&\displaystyle 
\varepsilon^i(r)= r^* \chi^i_R(r) = \left(\sum_{\scriptscriptstyle \alpha} r^*_\alpha \bfs^\alpha\right) \left(\sum_{\scriptscriptstyle \alpha} \alpha_i r_\alpha \bfs^\alpha\right)= \sum_{\scriptscriptstyle \alpha} \left(\sum_{\scriptscriptstyle \beta+\gamma=\alpha} \gamma_i r^*_\beta \, r_\gamma \right) \bfs^\alpha,
&\\
&\displaystyle 
\varepsilon(r)= r^* \bchi_R(r) = \left(\sum_{\scriptscriptstyle \alpha} r^*_\alpha \bfs^\alpha\right) \left(\sum_{\scriptscriptstyle \alpha} |\alpha| r_\alpha \bfs^\alpha\right)= \sum_{\scriptscriptstyle \alpha} \left(\sum_{\scriptscriptstyle \beta+\gamma=\alpha} |\gamma| r^*_\beta \, r_\gamma \right) \bfs^\alpha,
\end{eqnarray*}
and so, by using Lemma \ref{lemma:unit-psr}, we obtain:
\begin{eqnarray*}
&\displaystyle 
\varepsilon^i(r)= \sum_{\substack{\scriptscriptstyle \alpha\in\Delta\\ \scriptscriptstyle \alpha_i>0}}
\left( 
\sum_{\scriptscriptstyle d=1}^{\scriptscriptstyle |\alpha|} (-1)^{d-1} 
\left( \sum_{\scriptscriptstyle \alpha^\bullet\in \Par (\alpha,d) }
\alpha^d_i r_{\alpha^1}\, \cdots \, r_{\alpha^d}
\right) \right) \bfs^\alpha,
&\\
&\displaystyle 
 \varepsilon(r)= \sum_{\substack{\scriptscriptstyle \alpha\in\Delta\\ \scriptscriptstyle |\alpha|>0}}
\left( 
\sum_{\scriptscriptstyle d=1}^{\scriptscriptstyle |\alpha|} (-1)^{d-1} 
\left( \sum_{\scriptscriptstyle \alpha^\bullet\in \Par (\alpha,d) }
|\alpha^d| r_{\alpha^1}\, \cdots \, r_{\alpha^d}
\right) \right) \bfs^\alpha.
\end{eqnarray*}
In a similar way we obtain:
\begin{eqnarray*}
&\displaystyle 
\overline{\varepsilon}^i(r)= \sum_{\substack{\scriptscriptstyle \alpha\in\Delta\\ \scriptscriptstyle \alpha_i>0}}
\left( 
\sum_{\scriptscriptstyle d=1}^{\scriptscriptstyle |\alpha|} (-1)^{d-1} 
\left( \sum_{\scriptscriptstyle \alpha^\bullet\in \Par (\alpha,d) }
\alpha^1_i r_{\alpha^1}\, \cdots \, r_{\alpha^d}
\right) \right) \bfs^\alpha,
&\\
&\displaystyle 
 \overline{\varepsilon}(r)= \sum_{\substack{\scriptscriptstyle \alpha\in\Delta\\ \scriptscriptstyle |\alpha|>0}}
\left( 
\sum_{\scriptscriptstyle d=1}^{\scriptscriptstyle |\alpha|} (-1)^{d-1} 
\left( \sum_{\scriptscriptstyle \alpha^\bullet\in \Par (\alpha,d) }
|\alpha^1| r_{\alpha^1}\, \cdots \, r_{\alpha^d}
\right) \right) \bfs^\alpha.
\end{eqnarray*}
In particular, we have $\varepsilon^i_\alpha(r) = \overline{\varepsilon}^i_\alpha(r) =0$ whenever $\alpha_i=0$, i.e. whenever $i\notin \supp \alpha$, and $\varepsilon_0(r) = \overline{\varepsilon}_0(r) =0$:
\begin{eqnarray*}
&\displaystyle 
 \varepsilon^i(r)= \sum_{\scriptscriptstyle i\in \supp \alpha} \varepsilon^i_\alpha(r) \bfs^\alpha,\quad 
 \overline{\varepsilon}^i(r)= \sum_{\scriptscriptstyle i\in \supp \alpha} \overline{\varepsilon}^i_\alpha(r) \bfs^\alpha,
 &\\
&\displaystyle 
\varepsilon(r)= \sum_{\scriptscriptstyle |\alpha|>0} \varepsilon_\alpha(r) \bfs^\alpha,\quad 
 \overline{\varepsilon}(r)= \sum_{\scriptscriptstyle |\alpha|>0} \overline{\varepsilon}_\alpha(r) \bfs^\alpha,
\end{eqnarray*}
and $\varepsilon^i(r), \overline{\varepsilon}^i(r), \varepsilon(r), \overline{\varepsilon}(r) \in R[[\bfs]]_{\Delta,+}$ (see (\ref{nume:M[[bfs]]Delta})). 
The following recursive identities hold:
\begin{eqnarray*}
&\displaystyle 
\alpha_i r_\alpha = \sum_{\substack{\scriptscriptstyle \beta + \gamma =\alpha\\ \scriptscriptstyle \gamma_i>0}} r_{\beta} \, \varepsilon^i_\gamma(r)= \sum_{\substack{\scriptscriptstyle \beta + \gamma =\alpha\\ \scriptscriptstyle \gamma_i>0}} \overline{\varepsilon}^i_\gamma(r) \, r_{\beta},
 &\\
&\displaystyle 
\varepsilon^i_\alpha(r)  = \alpha_i r_\alpha - \sum_{\substack{\scriptscriptstyle \beta + \gamma =\alpha \\ \scriptscriptstyle |\beta|,\gamma_i>0}} r_{\beta} \, \varepsilon^t_\gamma(r),\quad
\overline{\varepsilon}^t_\alpha(r)  = \alpha_t r_\alpha - \sum_{\substack{\scriptscriptstyle \beta + \gamma =\alpha \\ \scriptscriptstyle |\beta|,\gamma_i>0}}  \overline{\varepsilon}^i_\gamma(r) \, r_{\beta},
\end{eqnarray*}
for all $\alpha\in\Delta$ with $\alpha_i>0$, and:
\begin{equation} \label{eq:explicit-varepsilon}  
|\alpha| r_\alpha = \sum_{\substack{\scriptscriptstyle \beta + \gamma =\alpha\\ \scriptscriptstyle |\gamma|>0}} r_{\beta} \, \varepsilon_\gamma(r)= \sum_{\substack{\scriptscriptstyle \beta + \gamma =\alpha\\ \scriptscriptstyle |\gamma|>0}} \overline{\varepsilon}_\gamma(r) \, r_{\beta},
\end{equation}
\begin{equation*}
\varepsilon_\alpha(r)  = |\alpha| r_\alpha - \sum_{\substack{\scriptscriptstyle \beta + \gamma =\alpha \\ \scriptscriptstyle |\beta|,|\gamma|>0}} r_{\beta} \, \varepsilon_\gamma(r),\quad
\overline{\varepsilon}_\alpha(r)  = |\alpha| r_\alpha - \sum_{\substack{\scriptscriptstyle \beta + \gamma =\alpha \\ \scriptscriptstyle |\beta|,|\gamma|>0}}  \overline{\varepsilon}_\gamma(r) \, r_{\beta},
\end{equation*}
for all $\alpha\in\Delta$.

\begin{nota} After (\ref{eq:explicit-varepsilon}), our definition of $\overline{\varepsilon}$ generalizes the construction in \cite{Mirza_2010}. 
\end{nota}

\begin{lemma} \label{lemma:crossed-der-epsilon}
For any $r\in \U^p(R;\Delta)$ and any $i,j=1,\dots,p$ the following identity holds:
$$ \chi^j_R \left( \varepsilon^i(r) \right) -
\chi^i_R \left( \varepsilon^j(r) \right) = [\varepsilon^i(r),\varepsilon^j(r)].
$$
\end{lemma}

\begin{proof} Since $[\chi^i,\chi^j]=0$, it is a consequence of Lemma
\ref{lemma:first-properties-epsilon}, (iii).
\end{proof}

\begin{notacion} \label{notacion:Lambda-bepsilon}
Under the above conditions, we will denote by $ \HH^p(R;\Delta)$ the subset of $\left( R[[\bfs]]_{\Delta,+} \right)^p$ whose elements are the families $\{\delta^i\}_{1\leq i\leq p}$ satisfying the following properties:
\begin{itemize}
\item[(a)] If $\delta^i = \sum_{|\alpha|>0} \delta^i_\alpha \bfs^\alpha$, we have $\delta^i_\alpha=0$ whenever $\alpha_i=0$.
\item[(b)] For all $i,j=1,\dots,p$ we have $\chi^j_R \left( \delta^i \right) -
\chi^i_R \left( \delta^j \right) = [\delta^i,\delta^j]$.
\end{itemize}
Let us notice that property (b) may be explicitly written as:
\begin{equation} \label{eq:crossed-cond}
\alpha_j \delta^i_\alpha - \alpha_i \delta^j_\alpha = \sum_{\substack{\scriptscriptstyle \beta+\gamma=\alpha\\ \scriptscriptstyle \beta_i, \gamma_j>0 }} [\delta^i_\beta,\delta^j_\gamma]
\end{equation}
for all $i,j=1,\dots,p$ and for all $\alpha\in\Delta$ with $\alpha_i, \alpha_u>0$. 
Let us also consider the map:
$$ \bSigma: \{\delta^i\} \in \HH^p(R;\Delta) \longmapsto \sum_{i=1}^p \delta^i \in R[[\bfs]]_{\Delta,+}.
$$
After Lemma \ref{lemma:crossed-der-epsilon}, we can consider the map:
$$ \bepsilon: D \in \U^p(R;\Delta) \longmapsto \{ \varepsilon^i(r)\}_{1\leq i\leq p} \in \HH^p(R;\Delta)
$$
and we obviously have $\varepsilon = \bSigma \pcirc \bepsilon$.
\end{notacion}

\begin{prop} \label{prop:varepsilon-bijective-Q}
Assume that $\QQ \subset k$. Then, the three maps in the following commutative diagram:
$$
\begin{tikzcd}
\U^p(R;\Delta)   
 \ar[]{r}{\bepsilon} \ar{dr}{\varepsilon } & \HH^p(R;\Delta) \ar[]{d}{\bSigma} 
 \\
 & R[[\bfs]]_{\Delta,+}
\end{tikzcd}
$$
are bijective.
\end{prop}

\begin{proof} The injectivity of $\varepsilon$ is a straightforward consequence of (\ref{eq:explicit-varepsilon}).
Let us prove the surjectivity of $\varepsilon$.
Let $\overline{r} = \sum_\alpha \overline{r}_\alpha \bfs^\alpha$ be any element in $R[[\bfs]]_{\Delta,+}$. Since $\QQ \subset k$, the differential equation
$$ \bchi (Y) = Y \overline{r},\quad Y\in R[[\bfs]]_\Delta,
$$ 
has a unique solution $r\in R[[\bfs]]_\Delta$ with initial condition $r_0 = 1$, i.e. $r\in \U^p(R;\Delta)$.
It is given recursively by:
$$ |\alpha| r_\alpha = \sum_{\substack{\scriptscriptstyle \beta + \gamma=\alpha\\ \scriptscriptstyle  |\gamma|>0}} r_{\beta} \, \overline{r}_\gamma,\quad \alpha\in\Delta, |\alpha|>0,
$$
and so $\varepsilon(r)=\overline{r}$. 
To finish, the only missing point is the injectivity of $\bSigma$. Let $\{\delta^i\}, \{\eta^i\}\in \HH^p(R;\Delta)$ be with $\sum_i \delta^i = \sum_i \eta^i$. It is clear that $\delta^i_\alpha = \eta^i_\alpha$ whenever $|\alpha|=1$. Assume that $\delta^i_\beta = \eta^i_\beta$  for all $i=1,\dots,p$ whenever $|\beta|< m$ and consider $\alpha\in\Delta$ with $|\alpha|=m$. By using (\ref{eq:crossed-cond}) and the induction hypothesis we obtain:
\begin{equation*}
\alpha_j \delta^i_\alpha - \alpha_i \delta^j_\alpha = \sum_{\substack{\scriptscriptstyle \beta+\gamma=\alpha\\ \scriptscriptstyle \beta_i, \gamma_j>0 }} [\delta^i_\beta,\delta^j_\gamma] =
\sum_{\substack{\scriptscriptstyle \beta+\gamma=\alpha\\ \scriptscriptstyle \beta_i, \gamma_j>0 }} [\eta^i_\beta,\eta^j_\gamma] = \alpha_j \eta^i_\alpha - \alpha_i \eta^j_\alpha\quad \forall i,j\in \supp \alpha.
\end{equation*}
The above system of linear equations with rational coefficients joint with the linear equation:
\begin{eqnarray*}
&\displaystyle \sum_{\scriptscriptstyle i\in \supp \alpha} \delta^i_\alpha = \sum_{\scriptscriptstyle i\in \supp \alpha} \eta^i_\alpha,
\end{eqnarray*}
gives rise to a non singular system and we deduce that $\delta^i_\alpha = \eta^i_\alpha$ for all $i\in \supp \alpha$, and so $\delta^i_\alpha = \eta^i_\alpha$ for all $i=1,\dots,p$.
\end{proof}

Notice that Lemma \ref{lemma:crossed-der-epsilon} and Proposition \ref{prop:varepsilon-bijective-Q} can also be stated with the $\overline{\varepsilon}^i$ and $\overline{\varepsilon}$ instead of the $\varepsilon^i$ and $\varepsilon$.

\subsection{Substitution maps}

In this section we give a summary of sections 2 and 3 of \cite{nar_subst_2018}. Let $k$ be a commutative ring, $A$ a commutative $k$-algebra, $\bfs=\{s_1,\dots,s_p\},\bft=\{t_1,\dots,t_q\}$ two sets of variables and
$\Delta\subset \N^p, \nabla\subset \N^q$ non-empty co-ideals.
\medskip

\begin{defi}  \label{def:substitution_maps} 
An $A$-algebra map $\varphi:A[[\bfs]]_\Delta \xrightarrow{} A[[\bft]]_\nabla$
will be called a {\em substitution map} whenever $\ord (\varphi(s_i)) \geq 1$ for all $i=1,\dots, p$.
A such map is continuous and uniquely determined by the family $c=\{\varphi(s_i), i=1,\dots, p\}$.
\medskip

The {\em trivial} substitution map  $A[[\bfs]]_\Delta \xrightarrow{} A[[\bft]]_\nabla $ is the one sending any $s_i$ to $0$. It will be denoted by $\mathbf{0}$. 
\end{defi}

The composition of substitution maps is obviously a substitution map. Any substitution map $\varphi:A[[\bfs]]_\Delta \xrightarrow{} A[[\bft]]_\nabla$ determines and is determined by a family:
$$\left\{{\bf C}_e(\varphi,\alpha), e\in \nabla, \alpha\in\Delta, |\alpha|\leq |e|  \right\}\subset A,\quad \text{with\ }\ {\bf C}_0(\varphi,0)=1,
$$
such that:
\begin{equation} \label{eq:exp-substi}
\varphi\left( \sum_{\scriptscriptstyle\alpha\in\Delta} a_\alpha \bfs^\alpha \right) =
\sum_{\scriptscriptstyle e\in \nabla}
\left(  \sum_{\substack{\scriptscriptstyle \alpha\in\Delta\\ \scriptscriptstyle  |\alpha|\leq |e| }} {\bf C}_e(\varphi,\alpha) a_\alpha \right) \bft^e.
\end{equation}
In section 3, 2., of \cite{nar_subst_2018} the reader can find the explicit expression of the ${\bf C}_e(\varphi,\alpha)$ in terms of the $\varphi(s_i)$. 
The following lemma is clear.

\begin{lemma} \label{lemma:truncations-are-substitutions}
If $\Delta' \subset \Delta\subset \N^p$ are non-empty co-ideals, the truncation 
$\tau_{\Delta \Delta'} :  A[[\bfs]]_{\Delta} \to A[[\bfs]]_{\Delta'}$ is clearly a substitution map and 
${\bf C}_\beta\left(\tau_{\Delta \Delta'},\alpha\right) = \delta_{\alpha \beta}$ for all $\alpha\in\Delta$ and for all $\beta\in\Delta'$ with $|\alpha|\leq |\beta|$.
\end{lemma}

\begin{defi} We say that a substitution map $\varphi: A[[\bfs]]_\Delta \xrightarrow{} A[[\bft]]_\nabla $ has {\em constant coefficients} 
if $\varphi(s_i) \in k[[\bft]]_\nabla$ for all $i=1,\dots,p$.
This is equivalent to saying that
 ${\bf C}_e(\varphi,\alpha)\in k$ for all $e\in\nabla$ and for all $\alpha\in \Delta$ with $|\alpha|\leq |e|$. Substitution maps which constant coefficients are induced by substitution maps $k[[\bfs]]_\Delta \xrightarrow{} k[[\bft]]_\nabla$.
\end{defi}
\medskip

\numero  \label{sec:action-substi}
Let
$M$ be an $(A;A)$-bimodule.
\medskip

Any substitution map map $\varphi: A[[\bfs]]_\Delta \to A[[\bft]]_\nabla$ induces 
$(A;A)$-linear maps:  
$$\varphi_M := \varphi \widehat{\otimes} \Id_M : M[[\bfs]]_\Delta \equiv A[[\bfs]]_\Delta \widehat{\otimes}_A M \longrightarrow M[[\bft]]_\nabla \equiv A[[\bft]]_\nabla \widehat{\otimes}_A M
$$ 
and 
$$\sideset{_M}{}\opvarphi := \Id_M \widehat{\otimes} \varphi  : M[[\bfs]]_\Delta \equiv M \widehat{\otimes}_A A[[\bfs]]_\Delta \longrightarrow M[[\bft]]_\nabla \equiv M \widehat{\otimes}_A A[[\bft]]_\nabla.
$$
We have:
\begin{eqnarray*}
&\displaystyle \varphi_M \left(\sum_{ \scriptscriptstyle \alpha\in\Delta} m_\alpha \bfs^\alpha \right) = \sum_{\scriptscriptstyle \alpha\in\Delta} \varphi( \bfs^\alpha ) m_\alpha = \sum_{ \scriptscriptstyle e\in \nabla}   \left( \sum_{\substack{\scriptscriptstyle \alpha\in\Delta\\ \scriptscriptstyle  |\alpha|\leq |e| } }
{\bf C}_e(\varphi,\alpha) m_\alpha \right) \bft^e,
&\\
&\displaystyle 
\sideset{_M}{}\opvarphi \left(\sum_{ \scriptscriptstyle \alpha\in\Delta} m_\alpha \bfs^\alpha \right) = \sum_{\scriptscriptstyle \alpha\in\Delta}  m_\alpha \varphi( \bfs^\alpha ) = \sum_{ \scriptscriptstyle e\in\nabla }   \left( \sum_{ \substack{\scriptscriptstyle \alpha\in\Delta\\ \scriptscriptstyle  |\alpha|\leq |e| } } 
m_\alpha {\bf C}_e(\varphi,\alpha)  \right) \bft^e&
\end{eqnarray*}
for all $m\in M[[\bfs]]_\Delta$. 
If $M$ is a trivial bimodule, then $\varphi_M = \sideset{_M}{}\opvarphi$. 
If $\varphi': A[[\bft]]_\nabla \to A[[\bfu]]_\Omega$ is another substitution map and $\opvarphi'' = \opvarphi\pcirc \opvarphi'$, we have
$\opvarphi''_M = \varphi_M \pcirc \varphi'_M$, $\sideset{_M}{}\opvarphi'' = \sideset{_M}{}\opvarphi \pcirc \sideset{_M}{}\opvarphi'$.
\bigskip

For all $m\in M[[\bfs]]_\Delta$ and all $a \in A[[\bfs]]_\nabla$, we have:
$$ \varphi_M (a m) = \varphi(a) \varphi_M(m),\ \sideset{_M}{}\opvarphi (m a) =  \sideset{_M}{}\opvarphi (m) \varphi(a),
$$
i.e. $\varphi_M$ is $(\varphi;A)$-linear and $\sideset{_M}{}\opvarphi$ is $(A;\varphi)$-linear. Moreover, $\varphi_M$ and $\sideset{_M}{}\opvarphi$ are compatible with the augmentations, i.e.:
\begin{equation} \label{eq:compat-augment}
\varphi_M(m) \equiv m_0\!\!\!\!\mod \langle \bft \rangle M[[\bft]]_\nabla,\ 
\sideset{_M}{}\opvarphi(m) \equiv m_0\!\!\!\!\mod \langle \bft \rangle M[[\bft]]_\nabla,\  m\in M[[\bfs]]_\Delta.
\end{equation}
If $\varphi$ is the trivial substitution map (i.e. $\varphi(s_i)=0$ for all $s_i\in\bfs$), then $\varphi_M : M[[\bfs]]_\Delta \to M[[\bft]]_\nabla$ and $\sideset{_M}{}\opvarphi: M[[\bfs]]_\Delta \to M[[\bft]]_\nabla$ are also trivial, i.e.
$ \varphi_M (m) = \sideset{_M}{}\opvarphi
(m) = m_0$, for all $m\in M[[\bfs]]_\nabla$.
\bigskip

\numero \label{nume:def-sbullet}
The above constructions apply in particular to the case of any $k$-algebra $R$ over $A$, for which we have two induced continuous maps:
$\varphi_R= \varphi \widehat{\otimes} \Id_R : R[[\bfs]]_\Delta \to R[[\bft]]_\nabla$, which is $(A;R)$-linear, and  $\sideset{_R}{}\opvarphi= \Id_R \widehat{\otimes} \varphi  : R[[\bfs]]_\Delta \to R[[\bft]]_\nabla$, which is $(R;A)$-linear. 
For $r\in R[[\bfs]]_\Delta$ we will denote
$ \varphi \sbullet r := \varphi_R(r)$, $r \sbullet \varphi := \sideset{_R}{}\opvarphi(r)$. 
Explicitly, if $r=\sum_\alpha
r_\alpha \bfs^\alpha$ with $\alpha\in\Delta$, then:
\begin{equation} \label{eq:explicit-bullet}
\varphi \sbullet r
 = 
\sum_{\scriptscriptstyle e \in\nabla}
\left( \sum_{\substack{\scriptscriptstyle \alpha\in\Delta\\ \scriptscriptstyle  |\alpha|\leq |e|}} {\bf C}_e(\varphi,\alpha) r_\alpha \right) \bft^e,\quad
 r \sbullet \varphi 
 = 
\sum_{\scriptscriptstyle e \in\nabla}
\left( \sum_{\substack{\scriptscriptstyle \alpha\in\Delta\\ \scriptscriptstyle  |\alpha|\leq |e|}} r_\alpha {\bf C}_e(\varphi,\alpha)  \right) \bft^e.
\end{equation}
From (\ref{eq:compat-augment}), we deduce that:
$$\varphi \sbullet \U^p(R;\Delta) \subset \U^q(R;\nabla),\quad \U^p(R;\Delta)\sbullet \varphi \subset \U^q(R;\nabla),
$$ 
and $\varphi \sbullet 1 = 1 \sbullet \varphi = 1$. 
\medskip

\noindent
If $\varphi$ is a substitution map with constant coefficients, then $\varphi_R = \sideset{_R}{}\opvarphi$ is a ring homomorphism over $\varphi$. 
In particular, $\varphi \sbullet r = r \sbullet \varphi$ and $\varphi \sbullet (rr') = (\varphi \sbullet r) (\varphi \sbullet r')$.
\medskip

\noindent 
If $\varphi = \mathbf{0}: A[[\bfs]]_\Delta \to A[[\bft]]_\nabla$ is the trivial substitution map, then 
$\mathbf{0} \sbullet r = r\sbullet \mathbf{0} = r_0$ for all $r\in R[[\bfs]]_\Delta$. In particular,
$\mathbf{0} \sbullet r = r \sbullet \mathbf{0} = 1$ for all $r\in \U^p(R;\Delta)$.
\medskip

\noindent 
If $\bfu=\{u_1,\dots,u_r\}$ is another set of variables, $\Omega \subset \N^r$ is a non-empty co-ideal and
$\psi:R[[\bft]]_\nabla \to R[[\bfu]]_\Omega$ is another substitution map, one has:
$$ \psi \sbullet (\varphi \sbullet r) = (\psi \pcirc \varphi) \sbullet r,\quad
(r \sbullet \varphi) \sbullet \psi  =  r \sbullet (\psi \pcirc \varphi).
$$
Since $\left(R[[\bfs]]_\Delta\right)^{\text{opp}} = R^{\text{opp}}[[\bfs]]_\Delta$, for any substitution map $\varphi: A[[\bfs]]_\Delta \to A[[\bft]]_\nabla$  we have $\left( \varphi_R \right)^{\text{opp}} = \sideset{_{R^{\text{opp}}}}{}\opvarphi$ and 
$\left( \sideset{_R}{}\opvarphi \right)^{\text{opp}} = \varphi_{R^{\text{opp}}}$. 
\medskip

For each substitution map $\varphi: A[[\bfs]]_\Delta \to A[[\bft]]_\nabla$ we define the $(A;A)$-linear map:
$$ \varphi_*: f\in \Hom_k(A,A[[\bfs]]_\Delta) \longmapsto \varphi_*(f)= \varphi \pcirc f \in \Hom_k(A,A[[\bft]]_\nabla)
$$
which induces another one $\overline{\varphi_*}: \End_{k[[\bfs]]_\Delta}(A[[\bfs]]_\Delta) \longrightarrow
\End_{k[[\bft]]_\nabla}(A[[\bft]]_\nabla)$ given by:

$$ \overline{\varphi_*}(f):=
\left(\varphi_*\left(f|_A\right)\right)^e =
\left(\varphi \pcirc f|_A\right)^e\quad \forall f\in \End_{k[[\bfs]]_\Delta}(A[[\bfs]]_\Delta).
$$
More generally, for any left $A$-modules $E,F$ we have $(A;A)$-linear maps:
\begin{eqnarray*}
&\displaystyle
 (\varphi_{F})_*: f\in \Hom_k(E,F[[\bfs]]_\Delta) \longmapsto (\varphi_{F})_*(f)= \varphi_F \pcirc f \in \Hom_k(E,F[[\bft]]_\nabla),
&\\
&\displaystyle
\overline{(\varphi_F)_*}:  \Hom_{k[[\bfs]]_\Delta}(E[[\bfs]]_\Delta,F[[\bfs]]_\Delta) \longrightarrow  \Hom_{k[[\bft]]_\nabla}(E[[\bft]]_\nabla,F[[\bft]]_\nabla),
&\\
&\displaystyle
 \overline{(\varphi_F)_*}(f):=\left(\varphi_F \pcirc f|_E\right)^e.
\end{eqnarray*}
Let us consider the $(A;A)$-bimodule $M=\Hom_k(E,F)$. For each $m\in M[[\bfs]]_\Delta$ and for each $e\in E$ we have
$ \widetilde{\varphi_M(m)}(e) = \varphi_F\left(\widetilde{m}(e) \right)$, i.e.:
\begin{equation}   \label{eq:tilde-varphi-M-varphi-F}
\widetilde{\varphi_M(m)}|_E = \varphi_F \pcirc \left(\widetilde{m}|_E\right),
\end{equation}
or more graphically, the following diagram is commutative (see (\ref{eq:comple_formal_1})):
\begin{equation} \label{eq:commut-CD-varphi-bullet}
\begin{tikzcd}
M[[\bfs]]_\Delta 
\ar[d,"\varphi_M"'] \ar[r,"\sim", "m\mapsto \widetilde{m}"'] & 
\Hom_{k[[\bfs]]_\Delta}(E[[\bfs]]_\Delta,F[[\bfs]]_\Delta) \ar[r,"\sim", "\text{restr.}"']
\ar[d,"\overline{(\varphi_F)_*}"] & \Hom_k(E,F[[\bfs]]_\Delta) \ar[d,"(\varphi_F)_*"'] \\
M[[\bft]]_\nabla \ar[r,"\sim", "m\mapsto \widetilde{m}"'] &
\Hom_{k[[\bft]]_\nabla}(E[[\bft]]_\nabla,F[[\bft]]_\nabla) \ar[r,"\sim", "\text{restr.}"'] &
\Hom_k(E,F[[\bft]]_\nabla).
\end{tikzcd}
\end{equation}
In order to simplify notations, we will also write:
$$ \varphi \sbullet f := \overline{(\varphi_F)_*}(f)\quad \forall f \in \Hom_{k[[\bfs]]_\Delta}(E[[\bfs]]_\Delta,F[[\bfs]]_\Delta),
$$
and so we have
$ \widetilde{\varphi \sbullet m} = \varphi \sbullet \widetilde{m}$ for all $m\in M[[\bfs]]_\Delta$. 
Let us notice that $(\varphi \sbullet f)(e) = (\varphi_F \pcirc f)(e)$ for all $e\in E$, i.e.: 
\begin{equation} \label{eq:atencion-bullet}
\framebox{$(\varphi \sbullet f)|_E = (\varphi_F \pcirc f)|_E = \varphi_F \pcirc \left(f|_E\right)$, but in general $\varphi \sbullet f \neq  \varphi_F \pcirc f$.}
\end{equation}
If $\varphi = \mathbf{0}$ is the trivial substitution map, then for each $k$-linear map
$f=\sum_\alpha f_\alpha \bfs^\alpha: E \to E[[\bfs]]_\Delta$ (resp. $f=\sum_\alpha f_\alpha \bfs^\alpha \in\End_k(E)[[\bfs]]_\Delta \equiv \End_{k[[\bfs]]_\Delta}(E[[\bfs]]_\Delta)$), we have $\mathbf{0} \sbullet f = f\sbullet \mathbf{0} = f_0 \in \End_k(E) \subset \Hom_k(E,E[[\bfs]]_\Delta)$ (resp. $\mathbf{0} \sbullet f = f\sbullet \mathbf{0} = f_0^e = \overline{f_0} \in \End_{k[[\bfs]]_\Delta}(E[[\bfs]]_\Delta)$).
\bigskip

If $\varphi: A[[\bfs]]_\Delta \to A[[\bft]]_\nabla$ is a substitution map, we have:
\begin{eqnarray*}
&\displaystyle
\varphi \sbullet (af) = \varphi(a) \left( \varphi \sbullet f\right),\   (fa) \sbullet \varphi  =  \left(f \sbullet \varphi \right) \varphi(a)
\end{eqnarray*}
for all $a\in A[[\bfs]]_\Delta$ and for all  $f\in \Hom_k(E,E[[\bfs]]_\Delta)$ (or $f\in \End_{k[[\bfs]]_\Delta}(E[[\bfs]]_\Delta)$). Moreover:
\begin{eqnarray*}
& \displaystyle (\varphi_E)_*(\Hom_k^\pcirc(E,M[[\bfs]]_\Delta)) \subset \Hom_k^\pcirc(E,E[[\bft]]_\nabla),
&
\\
& \displaystyle 
 \varphi \sbullet  \left( \Aut_{k[[\bfs]]_\Delta}^\pcirc(E[[\bfs]]_\Delta) \right) \subset
\Aut_{k[[\bft]]_\nabla}^\pcirc(E[[\bft]]_\nabla),&
\end{eqnarray*}
and so we have a commutative diagram:
\begin{equation}   \label{eq:diag-funda}
\begin{tikzcd}
\U^p(R;\Delta) \ar[r,"\sim", "r\mapsto \widetilde{r}"'] \ar[d,"\varphi \sbullet (-)"']  &   
 \Aut_{k[[\bfs]]_\Delta}^\pcirc(E[[\bfs]]_\Delta) \ar[d,"\varphi \sbullet (-)"] \ar[r,"\sim", "\text{restr.}"']  &
\Hom_k^\pcirc(E,E[[\bfs]]_\Delta)  \ar[d,"(\varphi_{E})_*"] \\
\U^q(R;\nabla) \ar[r,"\sim", "r\mapsto \widetilde{r}"'] & \Aut_{k[[\bft]]_\nabla}^\pcirc(E[[\bft]]_\nabla)
\ar[r,"\sim", "\text{restr.}"'] & \Hom_k(E,F[[\bft]]_\nabla).
\end{tikzcd}
\end{equation}
\bigskip

Now we are going to see how the $\varepsilon^i(r), \overline{\varepsilon}^i(r), \varepsilon(r),  \overline{\varepsilon}(r)$ (see \ref{nume:explicit_varepsilon}) can be expressed in terms of the action of substitution maps.
\medskip

Let us consider the power series ring $A[[\bfs,\tau]] = A[[\bfs]] \widehat{\otimes}_A A[[\tau]]$, and for each
 $i=1,\dots,p$ we denote $\sigma^i: A[[\bfs]] \to A[[\bfs,\tau]]$ the substitution map (with constant coefficients) defined by:
$$ \sigma^i(s_j) = \left\{ \begin{array}{lll}
s_i+s_i \tau & \text{if} & j=i\\
s_j & \text{if} & j\neq i.
\end{array} \right.
$$
Let us also denote $\sigma: A[[\bfs]] \to A[[\bfs,\tau]]$ the substitution map (with constant coefficients) defined by:
$$ \sigma(s_i) = s_i + s_i \tau\ \ \forall i=1,\dots,p, 
$$
and $\iota: A[[\bfs]] \to A[[\bfs,\tau]]$ the substitution map induced by the inclusion $\bfs \hookrightarrow \bfs \sqcup \{\tau\}$. We often consider $\iota$ as an inclusion $A[[\bfs]] \hookrightarrow A[[\bfs,\tau]]$.
\medskip

It is clear that for each non-empty co-ideal $\Delta \subset \N^p$, the substitution maps $\sigma^i,\sigma, \iota: A[[\bfs]] \to A[[\bfs,\tau]]$ induce new substitution maps $A[[\bfs]]_\Delta \to A[[\bfs,\tau]]_{\Delta \times \{0,1\}}$, which will be also denoted by the same letters. 
Moreover, as a consequence of Taylor's expansion we have:
$$ \sigma^i (a) = a + \chi^i_A(a) \tau,\quad \sigma (a) = a + \chi_A(a) \tau 
$$
where $\chi^i = s_i \frac{\partial}{\partial s_i}$ and $\chi = \sum_{i} \chi^i$ (see Definition\ref{def:Euler-derivation}).
\medskip

The proof of the following lemma is clear. 

\begin{lemma} \label{lemma:xi}
The map $\xi: R[[\bfs]]_{\Delta,+} \to \U^{p+1}(R;\Delta \times \{0,1\})$  defined as:
$$ \xi \left( \sum_{\scriptscriptstyle \alpha\in\Delta, |\alpha|>0} r_\alpha \bfs^\alpha \right) =
1 + \sum_{\scriptscriptstyle \alpha\in\Delta, |\alpha|>0} r_\alpha \bfs^\alpha \tau
$$
is a group homomorphism.
\end{lemma}

Let us notice that the map $\xi$ above is injective and its image is the set of $r\in \U^{p+1}(R;\Delta \times \{0,1\})$ such that $\supp r \subset \{(0,0)\} \cup ((\Delta \setminus \{0\}) \times \{1\})$.

\begin{prop} \label{prop:xi-varepsilon}
For each $r\in \U^p(R;\Delta)$, the following properties hold:
\begin{enumerate}
\item[(1)] $ r^* (\sigma^i\sbullet r)  = \xi( \varepsilon^i(r))$, $  (\sigma^i\sbullet r) r^*  = \xi( \overline{\varepsilon}^i(r))$.
\item[(2)] $ r^* (\sigma\sbullet r)  = \xi( \varepsilon(r))$, $ (\sigma\sbullet r) r^*   = \xi( \overline{\varepsilon}(r))$.
\end{enumerate}
\end{prop}

\begin{proof} It is a straightforward consequence of Taylor's expansion formula:
$$ \sigma^i\sbullet r = r + \chi^i_R(r) \tau,\quad \sigma \sbullet r = r + \chi_R(r) \tau.
$$
\end{proof}

Let us notice that in the above proposition, the action $\iota \sbullet (-): R[[\bfs]]_\Delta \to R[[\bfs,\tau]]_{\Delta \times \{0,1\}}$ is simply considered as an inclusion.

\section{Multivariate Hasse--Schmidt derivations}

\subsection{Basic definitions}
\label{section:HS}

In this section we recall some notions and results of the theory of Hasse--Schmidt derivations \cite{has37,mat_86} as developed in \cite{nar_subst_2018}. 
\medskip

From now on $k$ will be a commutative ring, $A$ a commutative $k$-algebra, $\bfs =\{s_1,\dots,s_p\}$ a set of variables and $\Delta \subset \N^p$ a non-empty co-ideal.

\begin{defi} 
A {\em $(p,\Delta)$-variate Hasse--Schmidt derivation}, or 
a {\em $(p,\Delta)$-va\-riate HS-de\-ri\-va\-tion} for short, of $A$ over $k$  
 is a family $D=(D_\alpha)_{\alpha\in \Delta}$ 
 of $k$-linear maps $D_\alpha:A
\longrightarrow A$, satisfying the following Leibniz type identities: $$ D_0=\Id_A, \quad
D_\alpha(xy)=\sum_{\beta+\gamma=\alpha}D_\beta(x)D_\gamma(y) $$ for all $x,y \in A$ and for all
$\alpha\in \Delta$. We denote by
$\HS^p_k(A;\Delta)$ the set of all $(p,\Delta)$-variate HS-derivations of $A$ over
$k$. For $p=1$, a $1$-variate HS-derivation will be simply called a {\em Hasse--Schmidt derivation}   (a HS-derivation for short), or a {\em higher derivation}\footnote{This terminology is used for instance in \cite{mat_86}.}, and we will simply write $\HS_k(A;m):= \HS^1_k(A;\Delta)$ for $\Delta=\{q\in\N\ |\ q\leq m\}$.\footnote{These HS-derivations are called of length $m$ in \cite{nar_2012}.}
\end{defi}

Any $(p,\Delta)$-variate HS-derivation $D$ of $A$ over $k$  can be understood as a power series:
$$ \sum_{\scriptscriptstyle \alpha\in\Delta} D_\alpha \bfs^\alpha \in R[[\bfs]]_\Delta,\quad R=\End_k(A),$$
and so we consider $\HS^p_k(A;\Delta) \subset R[[\bfs]]_\Delta$. Actually 
$\HS^p_k(A;\Delta)$ is a (multiplicative) sub-group of $\U^p(R;\Delta)$. 
The group operation in $\HS^p_k(A;\Delta)$ is explicitly given by:
$$ (D,E) \in \HS^p_k(A;\Delta) \times \HS^p_k(A;\Delta) \longmapsto D \pcirc E \in \HS^p_k(A;\Delta)$$
with:
$$ (D \pcirc E)_\alpha = \sum_{\scriptscriptstyle \beta+\gamma=\alpha} D_\beta \pcirc E_\gamma,
$$
and the identity element of $\HS^p_k(A;\Delta)$ is $\mathbb{I}$ with $\mathbb{I}_0 = \Id$ and 
$\mathbb{I}_\alpha = 0$ for all $\alpha \neq 0$. The inverse of a $D\in \HS^p_k(A;\Delta)$ will be denoted by $D^*$. 
\medskip

For $\Delta' \subset \Delta \subset \N^p$ non-empty co-ideals, we have truncations:
$$\tau_{\Delta \Delta'}: \HS^p_k(A;\Delta) \longrightarrow \HS^p_k(A;\Delta'),
$$ which obviously are group homomorphisms. 
Since any $D\in\HS_k^p(A;\Delta)$ is determined by its finite truncations, we have a natural group isomorphism
\begin{equation} \label{eq:HS-inv-limit-finite}
 \HS_k^p(A) =  \lim_{\stackrel{\longleftarrow}{\substack{\scriptscriptstyle \Delta' \subset \Delta\\ \scriptscriptstyle  \sharp \Delta'<\infty}}} \HS_k^p(A;\Delta').
\end{equation}

The proof of the following proposition is straightforward and it is left to the reader (see Notation \ref{notacion:Ump} and \ref{nume:widetilde-r}).

\begin{prop} \label{prop:caracter_HS}
Let us denote $R=\End_k(A)$ and
let $D= \sum_\alpha D_\alpha \bfs^\alpha \in R[[\bfs]]_\Delta$ be a power series. The following properties are equivalent:
\begin{enumerate}
\item[(a)] $D$ is a $(\bfs,\Delta)$-variate HS-derivation of $A$ over $k$.
\item[(b)] The map $\widetilde{D}: A[[\bfs]]_\Delta \to A[[\bfs]]_\Delta$ is a (continuous) $k[[\bfs]]_\Delta$-algebra homomorphism compatible with the natural augmentation $A[[\bfs]]_\Delta \to A$.
\item[(c)] $D\in\U^p(R;\Delta)$ and for all $a\in A[[\bfs]]_\Delta$ we have $D a  = \widetilde{D}(a) D$.
\item[(d)] $D\in\U^p(R;\Delta)$ and for all  $a\in A$ we have $D a  = \widetilde{D}(a) D$.
\end{enumerate}
Moreover, in such a case $\widetilde{D}$ is a $k[[\bfs]]_\Delta$-algebra automorphism of $A[[\bfs]]_\Delta$.
\end{prop}

\begin{notacion} \label{notacion:pcirc-HS}
Let us denote:
\begin{eqnarray*}
&\displaystyle  \Hom_{k-\text{\rm alg}}^\pcirc(A,A[[\bfs]]_\Delta) :=
\left\{ f \in \Hom_{k-\text{\rm alg}}(A,A[[\bfs]]_\Delta),\ f(a) \equiv a\!\!\!\!\mod \langle \bfs \rangle\  \forall a\in A \right\}, &\\
&\displaystyle  \Aut_{k[[\bfs]]_\Delta-\text{\rm alg}}^\pcirc(A[[\bfs]]_\Delta) :=
&\\
&\displaystyle 
 \left\{ f \in \Aut_{k[[\bfs]]_\Delta-\text{\rm alg}}(A[[\bfs]]_\Delta)\ |\ f(a) \equiv a_0\!\!\!\!\mod \langle \bfs \rangle\ \forall a\in A[[\bfs]]_\Delta \right\}.
\end{eqnarray*}
\end{notacion}
It is clear that $\Hom_{k-\text{\rm alg}}^\pcirc(A,A[[\bfs]]_\Delta) \subset \Hom_k^\pcirc(A,A[[\bfs]]_\Delta)$ and 
$$\Aut_{k[[\bfs]]_\Delta-\text{\rm alg}}^\pcirc(A[[\bfs]]_\Delta) \subset \Aut_{k[[\bfs]]_\Delta}^\pcirc(A[[\bfs]]_\Delta)$$ (see Notation \ref{notacion:pcirc}) are subgroups, and
we have  group isomorphisms (see (\ref{eq:Aut-iso-pcirc}) and  (\ref{eq:U-iso-pcirc})):

\begin{equation} \label{eq:HS-funda}
\begin{tikzcd}
\HS^p_k(A;\Delta) \ar[r,"D \mapsto \widetilde{D}","\simeq"'] & \Aut_{k[[\bfs]]_\Delta-\text{\rm alg}}^\pcirc(A[[\bfs]]_\Delta)  \ar[r,"\text{restr.}","\simeq"'] & \Hom_{k-\text{\rm alg}}^\pcirc(A,A[[\bfs]]_\Delta).
\end{tikzcd}
\end{equation}
The composition of the above isomorphisms is given by:
\begin{equation} \label{eq:HS-iso-Hom(A,A[[s]])}
 D\in \HS^p_k(A;\Delta) \stackrel{\sim}{\longmapsto} \Phi_D := \left[a\in A \mapsto \sum_{\scriptscriptstyle \alpha\in\Delta} D_\alpha(a) \bfs^\alpha\right] \in \Hom_{k-\text{\rm alg}}^\pcirc(A,A[[\bfs]]_\Delta).
\end{equation}
For each HS-derivation $D\in \HS^p_k(A;\Delta)$ we have $\widetilde{D} = \left(\Phi_D\right)^e$, i.e.:
$$ \widetilde{D}\left( \sum_{\scriptscriptstyle \alpha\in\Delta} a_\alpha \bfs^\alpha \right) = 
\sum_{\scriptscriptstyle \alpha\in\Delta} \Phi_D(a_\alpha) \bfs^\alpha\quad 
$$
for all $\sum_\alpha a_\alpha \bfs^\alpha \in A[[\bfs]]_\Delta$, and for any $E\in \HS^p_k(A;\Delta)$ we have $\Phi_{D \smallcirc E} = \widetilde{D} \pcirc \Phi_E$. 
If $\Delta'\subset \Delta$ is another non-empty co-ideal and
 we denote by  $\pi_{\Delta \Delta'}: A[[\bfs]]_\Delta \to A[[\bfs]]_{\Delta'}$ the projection (or truncation), one has $\Phi_{\tau_{\Delta \Delta'} (D)} = \pi_{\Delta \Delta'}\pcirc \Phi_D$.

\subsection{The action of substitution maps on HS-derivations}

Now, we recall the action of substitution maps on HS-derivations \cite[\S 6]{nar_subst_2018}.
Let $\bfs =\{s_1,\dots,s_p\}$, $\bft =\{t_1,\dots,t_p\}$ be sets of variables, $\Delta \subset \N^p$, $\nabla \subset \N^q$ non-empty co-ideals and let us write $R=\End_k(A)$.
\medskip

Let us recall Proposition 10 in \cite{nar_subst_2018}.

\begin{prop} \label{prop:equiv-action-subs-HS}
For any substitution map $\varphi:A[[\bfs]]_\Delta \to A[[\bft]]_\nabla$, we have:
\begin{enumerate}
\item[1)] $\varphi_*\left(\Hom_{k-\text{\rm alg}}^\pcirc(A,A[[\bfs]]_\Delta)\right) \subset \Hom_{k-\text{\rm alg}}^\pcirc(A,A[[\bft]]_\nabla)$,
\item[2)] $\varphi \sbullet \HS^p_k(A;\Delta) \subset \HS^q_k(A;\nabla)$,
\item[3)] $ \varphi\sbullet \Aut_{k[[\bfs]]_\Delta-\text{\rm alg}}^\pcirc(A[[\bfs]]_\Delta)  \subset
\Aut_{k[[\bft]]_\nabla-\text{\rm alg}}^\pcirc(A[[\bft]]_\nabla)$.
\end{enumerate}
\end{prop}

Then we have a commutative diagram:
\begin{equation} \label{eq:diag-funda-HS}
\begin{tikzcd}
\Hom_{k-\text{\rm alg}}^\pcirc(A,A[[\bfs]]_\Delta) \ar[d,"\varphi_*"'] & \HS^p_k(A;\Delta) \ar[l,"\sim"', "\Phi_D \mapsfrom D"] \ar[r,"\sim"] \ar[d,"\varphi\sbullet (-)"] & \Aut_{k[[\bfs]]_\Delta-\text{\rm alg}}^\pcirc(A[[\bfs]]_\Delta) n\ar[d,"\varphi\sbullet (-)"] \\
\Hom_{k-\text{\rm alg}}^\pcirc(A,A[[\bft]]_\nabla)  & \HS^q_k(A;\nabla) \ar[l,"\sim"', "\Phi_D \mapsfrom D"]
\ar[r,"\sim"] &
\Aut_{k[[\bft]]_\nabla-\text{\rm alg}}^\pcirc(A[[\bft]]_\nabla).
\end{tikzcd}
\end{equation}
In particular, for any HS-derivation $D\in \HS^p_k(A;\Delta)$ we have $\varphi \sbullet D\in \HS^q_k(A;\nabla)$ (see \ref{nume:def-sbullet}). Moreover
$ \Phi_{\varphi \sbullet D} = \varphi \pcirc \Phi_D$.
\medskip

It is clear that for any co-ideals $\Delta' \subset \Delta$ and $\nabla' \subset \nabla$ with $\varphi \left( \Delta'_A/\Delta_A \right) \subset \nabla'_A/\nabla_A$
we have:
\begin{equation} \label{eq:trunca_bullet}
\tau_{\nabla \nabla'}(\varphi \sbullet D) = \varphi' \sbullet \tau_{\Delta \Delta'}(D),
\end{equation}
where $\varphi': A[[\bfs]]_{\Delta'} \to A[[\bft]]_{\nabla'}$ is the substitution map induced by $\varphi$.
\bigskip

\numero
\label{nume:def-sbullet-HS} Let $\bfu=\{u_1,\dots,u_r\}$ be another set of variables, $\Omega\subset \N^r$ a non-empty co-ideal, $\varphi: A[[\bfs]]_\Delta \to A[[\bft]]_\nabla$, $\psi:  A[[\bft]]_\nabla \to  A[[\bfu]]_\Omega$ substitution maps and $D,D'\in\HS_k^p(A;\Delta)$ HS-derivations. From \ref{nume:def-sbullet} we deduce the following properties:\medskip

\noindent -) If we denote $E:= \varphi \sbullet D \in\HS_k^q(A;\nabla)$,  we have
\begin{equation} \label{eq:expression_phi_D} 
E_0=\Id,\quad  E_e = \sum_{\substack{\scriptstyle \alpha\in\Delta\\ \scriptstyle |\alpha|\leq |e| }} {\bf C}_e(\varphi,\alpha) D_\alpha,\quad \forall e\in \nabla.
\end{equation}
-) If $\varphi = \mathbf{0}$ is the trivial substitution map or if $D=\mathbb{I}$, then 
$\varphi \sbullet D = \mathbb{I}$.
\medskip

\noindent
-) If $\varphi$ has constant coefficients, then $\varphi \sbullet (D\pcirc D') = (\varphi \sbullet D) \pcirc (\varphi \sbullet D')$ and $(\varphi \sbullet D)^* = \varphi \sbullet D^*$
 (the general case is treated in Proposition \ref{prop:varphi-D-main}).
\medskip

\noindent 
-) $ \psi \sbullet (\varphi \sbullet D) = (\psi \pcirc \varphi) \sbullet D$.
\bigskip

The following result is proven in Propositions 11 and 12 of \cite{nar_subst_2018}.

\begin{prop} \label{prop:varphi-D-main}
Let $\varphi:A[[\bfs]]_\Delta  \to A[[\bft]]_\nabla$ be a substitution map. Then, 
the following assertions hold:
\begin{enumerate}
\item[(i)] For each $D\in\HS_k^p(A;\Delta)$ there is a unique substitution map
$\varphi^D: A[[\bfs]]_\Delta  \to A[[\bft]]_\nabla$ such that 
$\left(\widetilde{\varphi \sbullet D}\right) \pcirc \varphi^D =  \varphi \pcirc \widetilde{D}$. 
Moreover, $\left(\varphi\sbullet D\right)^* = \varphi^D \sbullet D^*$, $\varphi^{\mathbb{I}} = \varphi$ and:
$$ {\bf C}_e(\varphi,f+\nu) = \sum_{\substack{\scriptscriptstyle\beta+\gamma=e\\ 
\scriptscriptstyle |f+g|\leq |\beta|,|\nu|\leq |\gamma| }} {\bf C}_\beta(\varphi,f+g) D_g({\bf C}_\gamma(\varphi^D,\nu)) 
$$
for all $e\in \Delta$ and for all $f,\nu\in\nabla$ with $|f+\nu|\leq |e|$.
\item[(ii)] For each $D,E\in\HS_k^p(A;\nabla)$, we have $\varphi \sbullet (D \pcirc E) = (\varphi \sbullet D) \pcirc (\varphi^D \sbullet E)$ and 
 $\left(\varphi^D\right)^E = \varphi^{D \pcirc E}$. 
 In particular, $\left(\varphi^D\right)^{D^*} = \varphi$. 
\item[(iii)] If $\psi$ is another composable substitution map, then $(\varphi \pcirc \psi)^D = \varphi^{\psi\sbullet D} \pcirc \psi^D$.
\item[(iv)] If $\varphi$ has constant coefficients then $\varphi^D = \varphi$.
\end{enumerate}
\end{prop}

\section{Main results}

\subsection{The derivations associated with a Hasse--Schmidt derivation}

In this section $k$ will be a commutative ring, $A$ a commutative $k$-algebra, $R=\End_k(A)$, $\bfs = \{s_1,\dots,s_p\}$ a set of variables and $\Delta \subset \N^p$ a non-empty co-ideal.

\begin{lemma} \label{lemma:leibniz-<>-HS}
Let $\fd:k[[\bfs]]_\Delta \to k[[\bfs]]_\Delta$ be a $k$-derivation and $D\in \HS^p_k(A;\Delta)$ a HS-derivation. Then, for each $a\in A[[\bfs]]_\Delta$ we have
$\fd_R(D) a = \widetilde{\fd_R(D)}(a) D + \widetilde{D}(a) \fd_R(D)$.
\end{lemma}

\begin{proof} By using Lemma \ref{lemma:leibniz-<>} we have:
\begin{eqnarray*}
&
 \fd_R(D) a = \fd_R (D a) - D \fd_A (a) = \fd_R (\widetilde{D}(a) D) - \widetilde{D}(\fd_A (a)) D = &
\\  
& \fd_A (\widetilde{D}(a)) D +
\widetilde{D}(a) \fd_R (D) - \widetilde{D}(\fd_A (a)) D=
\widetilde{\fd_R (D)}(a)D+  \widetilde{D}(a) \fd_R (D).
& 
\end{eqnarray*}
\end{proof}

From now on, we will denote to simplify  $\fd_R = \fd$ and $\fd_A= \fd$. 

\begin{prop} \label{prop:exist-epsilon(D)} 
Let $\fd:k[[\bfs]]_\Delta \to k[[\bfs]]_\Delta$ be a $k$-derivation. Then, for any $D\in\HS^p_k(A;\Delta)$ we have
$\varepsilon^\fd(D), \overline{\varepsilon}^\fd(D)\in \Der_k(A)[[\bfs]]_{\Delta,+}=\Der_k(A)[[\bfs]]_\Delta \cap R[[\bfs]]_{\Delta,+}$.
\end{prop}

\begin{proof} Remember that $\overline{\varepsilon}^\fd(D)  = \fd(D) D^*$ and $ \varepsilon^\fd(D) =
D^* \fd(D)$ (Definition \ref{defi:varepsilon-fd}).
We will use Lemma \ref{lemma:carac-der-bfs} and Lemma \ref{lemma:leibniz-<>-HS}. For any $a\in A[[\bfs]]_\Delta$ we have:
\begin{eqnarray*}
&  \left(\fd(D) D^*\right) a = \fd(D) \widetilde{D^*}(a) D^* = 
\left[ \widetilde{\fd(D)}(\widetilde{D^*}(a))D + \widetilde{D}(\widetilde{D^*}(a)) \fd(D) \right]D^*=
&
\\
& \widetilde{\fd(D)}(\widetilde{D^*}(a))  + a \fd(D) D^* =
\widetilde{\left(\fd(D) D^*\right)}(a) + a \left(\fd(D) D^*\right),
\end{eqnarray*}
and so
$ [\fd(D) D^*, a] = \widetilde{(\fd(D) D^*)}(a)
$
and $\fd(D) D^* \in\Der_k(A)[[\bfs]]_\Delta$. The proof for $\varepsilon^\fd(D)$ is completely similar.
\end{proof}

\begin{exam} If $D\in \HS_k(A;m)$ is a $1$-variate HS-derivation of length $m$, then:
$$ \varepsilon_1(D)= D_1,\ \varepsilon_2(D) = 2 D_2 - D_1^2,\ \varepsilon_3(D) = 3 D_3 - 2 D_1 D_2 - D_2 D_1 + D_1^3,\dots
$$
\end{exam}

Let us recall that the map $\xi: R[[\bfs]]_{\Delta,+} \to \U^{p+1}(R;\Delta \times \{0,1\})$ has been defined in Lemma \ref{lemma:xi}. The proof of the following lemma is clear.

\begin{lemma} \label{lemma:xi-HS}
For each $\delta \in \Der_k(A)[[\bfs]]_{\Delta,+} $ we have $\xi(\delta) \in \HS^{p+1}_k(A;\Delta \times \{0,1\})$.  
\end{lemma}

So we have a group homomorphism $\xi: \Der_k(A)[[\bfs]]_{\Delta,+} \to \HS^{p+1}_k(A;\Delta \times \{0,1\})$ whose image is the set of $D\in \HS^{p+1}_k(A;\Delta \times \{0,1\})$ such that $\supp D \subset \{(0,0)\} \cup ((\Delta \setminus \{0\}) \times \{1\})$.
\medskip

The following proposition provides a characterization of HS-derivations in cha\-rac\-te\-ris\-tic $0$.

\begin{prop} \label{prop:characr-HS-Q}
Assume that $\QQ \subset k$, $R=\End_k(A)$ and $D\in \U^p(R;\Delta)$. The following properties are equivalent:
\begin{enumerate}
\item[(a)] $D\in \HS_k^p(A;\Delta)$.
\item[(b)] $\varepsilon^\fd(D) \in \Der_k(A)[[\bfs]]_\Delta$ for all $k$-derivations $\fd:k[[\bfs]]_\Delta \to k[[\bfs]]_\Delta$.
\item[(c)] $\varepsilon(D) \in \Der_k(A)[[\bfs]]_\Delta$.
\end{enumerate} 
\end{prop}

\begin{proof} The implication (a) $\Rightarrow$ (b) comes from Proposition \ref{prop:exist-epsilon(D)} and (b) $\Rightarrow$ (c) is obvious.
Let us prove (c) $\Rightarrow$ (a). Write $\delta = \varepsilon(D)$, i.e. $\bchi_R(D) = D\, \delta$. After Proposition \ref{prop:caracter_HS}, we need to prove that $D\, a= \widetilde{D}(a) D$ for all $a\in A$, and since $D\, a- \widetilde{D}(a) D$ belongs to the augmentation ideal of $R[[\bfs]]_\Delta$ and $\QQ\subset k$, it is enough to prove that $\bchi_R\left(D\, a- \widetilde{D}(a) D\right)=0$. By using that $\bchi_A(a)=0$ and Lemma \ref{lemma:leibniz-<>} we have:
\begin{eqnarray*}
& 
 \bchi_R\left(D\, a- \widetilde{D}(a) D\right) = \bchi_R(D)\, a + D\, \bchi_A(a) - \bchi_A(\widetilde{D}(a)) D - \widetilde{D}(a) \bchi_R(D) = &
\\
& D\, \delta\, a  - \widetilde{\bchi_R(D)}(a) D - \widetilde{D}(\bchi_A(a))D
- \widetilde{D}(a) D\, \delta =
&
\\
&
 D\, a\, \delta + D\, \widetilde{\delta}(a) - \widetilde{\bchi_R(D)}(a) D  -\widetilde{D}(a) D\, \delta =
&
\\
&
\widetilde{D}(a) D\, \delta + \widetilde{D}(\widetilde{\delta}(a)) D -\widetilde{(D\, \delta)}(a) D - \widetilde{D}(a) D\, \delta = 0.
\end{eqnarray*}
\end{proof}

\begin{thm} \label{thm:varepsilon-HS-Q-bijective}
Assume that $\QQ \subset k$. Then, all the maps in the following commutative diagram:
$$
\begin{tikzcd}
\HS^p_k(A;\Delta)  
 \ar[]{r}{\bepsilon} \ar{dr}{\varepsilon } & \HH^p(R;\Delta) \bigcap \left(\Der_k(A)[[\bfs]]_{\Delta,+}   \right)^p \ar[]{d}{\bSigma} 
 \\
 & \Der_k(A)[[\bfs]]_{\Delta,+}
\end{tikzcd}
$$
are bijective.
\end{thm}

\begin{proof} It is a consequence of Proposition \ref{prop:varepsilon-bijective-Q} and Proposition \ref{prop:characr-HS-Q}.
\end{proof}

A similar result holds for $\overline{\varepsilon}$ instead of $\varepsilon$.

\begin{nota} Let us notice that, in Theorem \ref{thm:varepsilon-HS-Q-bijective}, $\HS^p_k(A;\Delta)$ carries a group structure (non-commutative in general) and an action of substitution maps, and on the other hand $\Der_k(A)[[\bfs]]_{\Delta,+}$ carries an $A[[\bfs]]_\Delta$-module structure and a $k[[\bfs]]_\Delta$-Lie algebra structure, but the bijection $\varepsilon: \HS^p_k(A;\Delta) \xrightarrow{\sim} \Der_k(A)[[\bfs]]_{\Delta,+}$ is not compatible with these structures. The formulas expressing the behavior of $\varepsilon$ with respect to the group operation on HS-derivations or the behavior of $\varepsilon^{-1}$ with respect to the addition of power series with coefficients in $\Der_k(A)$, turn out to be complicated and have a similar flavor to Baker-Campbell-Hausdorff formula.
\end{nota}

\subsection{The behavior under the action of substitution maps}

\begin{defi} \label{def:D-element}
 Let $S$ be a $k$-algebra over $A$, $r\in\U^p(S;\Delta)$, $D\in\HS^p_k(A;\Delta)$, $r'\in S[[\bfs]]_\Delta$ and $\delta \in \Der_k(A)[[\bfs]]_\Delta$. We say that
\begin{enumerate}
\item[-)] $r$ is a {\em $D$-element}  if 
$ r\, a = \widetilde{D}(a)\,  r$ for all $a\in A[[\bfs]]_\Delta$. 
\item[-)] $r'$ is a {\em $\delta$-element} if
$r'\,  a = a r' + \widetilde{\delta}(a)\, 1_S$ for all $a\in A[[\bfs]]_\Delta$.
\end{enumerate} 
\end{defi}

It is clear that $D\in \HS^p_k(A;\Delta) \subset \U^p(\End_k(A);\Delta)$ is a $D$-element. For $D=\mathbb{I}$ the identity HS-derivation, a $r\in\U^p(S;\Delta)$ is an {\em $\mathbb{I}$-element} if and only if $r$ commutes with all $a\in A[[\bfs]]_\Delta$. If $E\in \HS^p_k(A;\Delta)$ is another HS-derivation, $r\in\U^p(S;\Delta)$ is a $D$-element and $s\in\U^p(S;\Delta)$ is an $E$-element, then $rs$ is a $(D\pcirc E)$-element. 
\medskip

The following lemma provides a characterization of $D$-elements. Its proof is easy and it is left to the reader.

\begin{lemma} \label{lemma:charact-D-elements}
With the above notations, 
for each $r=\sum_\alpha r_\alpha \bfs^\alpha\in\U^p(S;\Delta)$ the following properties are equivalent:
\begin{enumerate}
\item[-)] $r$ is a $D$-element.
\item[-)] $ b r  =  r\,  \widetilde{D^*}(b)$ for all $b\in A[[\bfs]]_\Delta$. 
\item[-)] $r^*$ is a $D^*$-element.
\item[-)] If $r=\sum_\alpha r_\alpha \bfs^\alpha$, we have $r_\alpha\,  a = \sum_{\beta+\gamma=\alpha} D_\beta(a) r_\gamma$ for all $a\in A$ and for all $\alpha \in\Delta$.
\item[-)] $ r\, a = \widetilde{D}(a) r$ for all $a\in A$.
\end{enumerate}
\end{lemma}

The following proposition reproduces Proposition 2.2.6 of \cite{nar_envelop}.

\begin{prop}   \label{prop:bullet-D-elements} 
Let $S$ be a $k$-algebra over $A$, $D\in\HS^p_k(A;\Delta)$, $\varphi:A[[\bfs]]_\Delta \to A[\bfu]]_\nabla$ a substitution map and $r\in\U^p(S;\Delta)$ a $D$-element. 
Then the following identities hold:
\begin{enumerate}
\item[(a)] $ \varphi_S(r)$ is a $(\varphi\sbullet D)$-element. 
\item[(b)] $\varphi_S(rr') = \varphi_S(r) \varphi^D_S(r')$ for all $r'\in R[[\bfs]]_\Delta$. In particular, 
$\varphi_S(r)^* = \varphi^D_S(r^*)$. 
\end{enumerate}
Moreover, if $E$ is an $A$-module and $S=\End_k(E)$, then the following identity holds:
\begin{enumerate}
\item[(c)] $\langle \varphi \sbullet r, \varphi^D_E(e) \rangle 
= \varphi_E \left( \langle r,e\rangle\right)$ for all $e\in E[[\bfs]]_\Delta$.
In other words:  $\varphi_E \pcirc \widetilde{r}=
\left(\varphi \sbullet \widetilde{r}\right) \pcirc \varphi^D_E$.
\end{enumerate}
\end{prop}

Lemma \ref{lemma:leibniz-<>} and Proposition \ref{prop:exist-epsilon(D)} can be generalized in the following way.

\begin{prop} \label{prop:epsilon-D-elements}
Let $S$ be a $k$-algebra over $A$, $D\in\HS^p_k(A;\Delta)$, $r\in\U^p(S;\Delta)$ a $D$-element and  $\fd:k[[\bfs]]_\Delta \to k[[\bfs]]_\Delta$ a $k$-derivation.
Then, the following properties hold:
\begin{enumerate}
\item[(1)]
$\fd(r)\, a = \langle \fd(D), a\rangle\,  r + \langle D, a\rangle\,  \fd(r)
$ for all $a\in A$
\item[(2)] $\varepsilon^\fd(r)$ is a $\varepsilon^\fd(D)$-element and  $\overline{\varepsilon}^\fd(r)$ is a $\overline{\varepsilon}^\fd(D)$-element.
\end{enumerate}
\end{prop}

\begin{proof} (1) Since $\delta(a)=0$ we have:
$$  \fd(r)\,  a = \delta( r\, a) =
 \fd \left(\widetilde{D}(a) r \right) =
 \fd\left(\langle D ,a\rangle\,  r\right) = \langle \fd(D) ,a\rangle\,  r +  \langle D ,a\rangle\,  \fd(r).
$$
(2) For all $a\in A$ we have:
\begin{eqnarray*}
& \varepsilon^\fd(r)\, a = r^*\, \fd(r)\, a \stackrel{\text{(1)}}{=} r^* \left(\widetilde{\fd(D)}(a)\, r +
\widetilde{D}(a)\, \fd(r)\right) = 
&
\\
&
\widetilde{D^*}(\widetilde{\fd(D)}(a)) r^*\, r + \widetilde{D^*}(\widetilde{D}(a)) r^* \fd(r)=
\widetilde{\varepsilon^\fd(D)}(a)\, 1_S + a\, \varepsilon^\fd(r).
\end{eqnarray*}
The proof for $\overline{\varepsilon}^\fd(r)$ is similar.
\end{proof}

Let us consider two sets of variables $\bfs=\{s_1,\dots,s_p\}$ and $\bfu=\{u_1,\dots,u_q\}$, and let us denote by $\{v^1,\dots,v^p\}$ the canonical basis of $\N^p$: $v^i_l = \delta_{il}$.

\begin{thm} \label{thm:epsilon-bullet}
For each non-empty co-ideals 
$\Delta\subset \N^p, \Omega \subset \N^q$, each substitution map $\varphi: A[[\bfs]]_\Delta \to A[[\bfu]]_\Omega$ and each HS-derivation $D\in \HS^p_k(A;\Delta)$, there exists a family 
$$\left\{ {\bf N}^{j,i}_{e,h}\ |\ 1\leq j\leq q, 1\leq i \leq p, e\in \Omega, h \in \Delta, |h|\leq |e|\right\} \subset A
$$ 
such that for any $k$-algebra $S$ over $A$ and any $D$-element $r\in \U^p(S;\Delta)$, we have:
\begin{equation}  \label{eq:varepsilon-varphi-of-D-element}
\varepsilon^j_e(\varphi\sbullet r) = \sum_{\substack{\scriptscriptstyle   0<|h|\leq |e|\\ \scriptscriptstyle i\in \supp h}}
{\bf N}^{j,i}_{e,h}
\varepsilon^i_h(r)\quad \forall e\in \Omega, \forall j=1,\dots,q. 
\end{equation}
Moreover, $\displaystyle {\bf N}^{j,i}_{e,h} = \sum_{\scriptscriptstyle f, g, \beta} g_j
{\bf C}_f(\varphi^D,\beta+h-v^i)D^*_\beta({\bf C}_g(\varphi,v^i))$, where $f,g\in \Omega$, $\beta \in \Delta$, $f + g = e$, $|\beta+h|-1\leq |f|$ and $g_j>0$, whenever $e_j, h_i > 0$, and ${\bf N}^{j,i}_{e,h} = 0$ otherwise.
\end{thm}

\begin{proof} Let us write $r= \sum_{\alpha\in\Delta} r_\alpha \bfs^\alpha$. For each $\alpha\in\Delta$ and each $j=1,\dots,q $ we have:
\begin{eqnarray*}
&\displaystyle
\chi^j (\varphi(\bfs^\alpha)) = \chi^j\left( \prod_{\scriptscriptstyle i=1}^{\scriptscriptstyle p} \varphi(s_i)^{\alpha_i} \right) = \sum_{\scriptscriptstyle \alpha_i\neq 0} \alpha_i \varphi(s_i)^{\alpha_i-1} \left( \prod_{\substack{\scriptscriptstyle 1\leq l\leq p\\ \scriptscriptstyle l\neq i}} \varphi(s_l)^{\alpha_l} \right) \chi^j(\varphi(s_i)) =
& \\
&\displaystyle
\sum_{\scriptscriptstyle \alpha_i\neq 0} \alpha_i \varphi\left(\bfs^{\alpha-v^i}\right)
\chi^j(\varphi(s_i)) = 
\sum_{\scriptscriptstyle \alpha_i\neq 0} \alpha_i \left(\sum_{\scriptscriptstyle |e|\geq |\alpha|-1}
{\bf C}_e(\varphi,\alpha-v^i) \bfu^e \right) \left( \sum_{\scriptscriptstyle e_j > 0} e_j {\bf C}_e(\varphi,v^i) \bfu^e \right)=
& \\
&\displaystyle
\sum_{\scriptscriptstyle \alpha_i\neq 0} \alpha_i \left(
\sum_{\substack{\scriptscriptstyle |e|\geq |\alpha|\\ \scriptscriptstyle e_j>0}} \left(
\sum_{\substack{\scriptscriptstyle e'+e''= e\\ \scriptscriptstyle |e'|\geq |\alpha|-1\\ \scriptscriptstyle e''_j> 0}} e''_j {\bf C}_{e'}(\varphi,\alpha-v^i) {\bf C}_{e''}(\varphi,v^i) \right)
\bfu^e \right) = 
\sum_{\scriptscriptstyle \alpha_i\neq 0} \alpha_i \left(
\sum_{\substack{\scriptscriptstyle |e|\geq |\alpha|\\ \scriptscriptstyle e_j> 0}} \mathbf{M}^{j,i}_{\alpha,e} \bfu^e \right)
\end{eqnarray*}
with:
$$ \mathbf{M}^{j,i}_{\alpha,e} :=  \sum_{\substack{\scriptscriptstyle e'+e''= e\\ \scriptscriptstyle |e'|\geq |\alpha|-1\\ \scriptscriptstyle e''_j>0}} e''_j {\bf C}_{e'}(\varphi,\alpha-v^i) {\bf C}_{e''}(\varphi,v^i)
$$
for 
$ i\in\supp \alpha$ and $e\in \Omega$ with $e_j>0$ and $|e|\geq |\alpha|$. If either $\alpha_i=0$ or $e_j=0$, we set $\mathbf{M}^{j,i}_{\alpha,e} = 0$. So, for each $j=1,\dots,q$ we have:
\begin{eqnarray*}
& \displaystyle
\varepsilon^j(\varphi\sbullet r) = (\varphi\sbullet r)^* \chi^j(\varphi\sbullet r) = \cdots =
(\varphi^D\sbullet r^*) \left(\sum_{\scriptscriptstyle \alpha \in\Delta}  \chi^j(\varphi(\bfs^\alpha)) r_\alpha\right) =
&
\\
& \displaystyle
\left(
\sum_{\scriptscriptstyle e\in\Omega} \left( 
\sum_{\scriptscriptstyle |\alpha|\leq |e|} {\bf C}_e(\varphi^D,\alpha)r^*_\alpha \right) \bfu^e \right)
\left( 
\sum_{\scriptscriptstyle e_j>0} \left(
 \sum_{\substack{\scriptscriptstyle |\alpha|\leq |e|\\ \scriptscriptstyle \alpha_i\neq 0}}
 \alpha_i \mathbf{M}^{j,i}_{\alpha,e} r_\alpha \right) \bfu^e
\right) =
&
\\
& \displaystyle
\sum_{\scriptscriptstyle e_j>0} \left(
 \sum_{\substack{\scriptscriptstyle f + g = e\\ \scriptscriptstyle 
|\mu|\leq |f|, |\nu|\leq |g|\\ \scriptscriptstyle g_j>0, i\in \supp \nu}}
\nu_t
{\bf C}_f(\varphi^D,\mu)r^*_\mu \mathbf{M}^{j,i}_{\nu,g} r_{\nu} 
\right) \bfu^e \stackrel{(\star)}{=}
&\\
&\displaystyle
\sum_{\scriptscriptstyle e_j>0} \left(
\sum_{\substack{\scriptscriptstyle f + g = e\\ \scriptscriptstyle 
|\beta+\gamma|\leq |f|, 0<|\nu|\leq |g|\\ \scriptscriptstyle g_j>0, i\in \supp \nu}}
\nu_i
{\bf C}_f(\varphi^D,\beta+\gamma)D^*_\beta( \mathbf{M}^{j,i}_{\nu,g} ) r^*_\gamma  r_\nu
\right) \bfu^e =
&\\
&\displaystyle
\sum_{\scriptscriptstyle e_j>0} \left(
\sum_{\substack{\scriptscriptstyle   0<|h|\leq |e|\\ \scriptscriptstyle   i\in\supp h}}
\sum_{\substack{\scriptscriptstyle \gamma+\nu=h\\ \scriptscriptstyle i\in \supp \nu}}
\left( \sum_{\substack{\scriptscriptstyle f + g = e\\ \scriptscriptstyle |\beta+\gamma|\leq |f|\\ \scriptscriptstyle 
 |\nu|\leq |g|, g_j>0}} 
{\bf C}_f(\varphi^D,\beta+\gamma)D^*_\beta(  \mathbf{M}^{j,i}_{\nu,g}  ) 
\right) 
\nu_i r^*_\gamma  r_\nu  \right) 
\bfu^e =
&\\
&\displaystyle
\sum_{\scriptscriptstyle e_j>0} \left(
\sum_{\substack{\scriptscriptstyle   0<|h|\leq |e|\\ \scriptscriptstyle   i\in\supp h}}
\sum_{\substack{\scriptscriptstyle \gamma+\nu=h\\ \scriptscriptstyle i\in \supp \nu}}
N^{j,i}_{e,\nu,\gamma}
\nu_i r^*_\gamma  r_\nu  \right) 
\bfu^e,
\end{eqnarray*}
where equality $(\star)$ comes from the fact that $r^*$ is a $D^*$-element (see Lemma \ref{lemma:charact-D-elements})
and
$$ N^{j,i}_{e,\nu,\gamma} = \left\{
\begin{array}{cl}\displaystyle 
\sum_{\substack{\scriptscriptstyle f + g = e\\ \scriptscriptstyle |\beta+\gamma|\leq |f|\\ \scriptscriptstyle 
 |\nu|\leq |g|,g_j>0}} 
{\bf C}_f(\varphi^D,\beta+\gamma)D^*_\beta(\mathbf{M}^{j,i}_{\nu,g}) & \text{if}\ 
\nu_i>0, e_j > 0, |e| \geq |\nu+\gamma|, \\
0 & \text{otherwise.}
\end{array}\right.
$$
But, for $h=\nu+\gamma$, we have:
\begin{eqnarray*}
& \displaystyle
N^{j,i}_{e,\nu,\gamma} = \sum_{\substack{\scriptscriptstyle f + g = e\\ \scriptscriptstyle |\beta+\gamma|\leq |f|\\ \scriptscriptstyle 
 |\nu|\leq |g|,g_j>0}} 
{\bf C}_f(\varphi^D,\beta+\gamma)D^*_\beta(\mathbf{M}^{j,i}_{\nu,g}) =
&
\\
&\displaystyle
\sum_{\substack{\scriptscriptstyle f + g = e\\ \scriptscriptstyle |\beta+\gamma|\leq |f|\\ \scriptscriptstyle 
 |\nu|\leq |g|,g_j>0}} 
{\bf C}_f(\varphi^D,\beta+\gamma)\, D^*_\beta \left(
 \sum_{\substack{\scriptscriptstyle g'+g''= g\\ \scriptscriptstyle |g'|\geq |\nu|-1\\ \scriptscriptstyle g''_j>0}} g''_j {\bf C}_{g'}(\varphi,\nu-v^i) {\bf C}_{g''}(\varphi,v^i)
\right) =
&
\end{eqnarray*}

\begin{eqnarray*}
& \displaystyle
\sum_{\substack{\scriptscriptstyle f + g'+g'' = e\\ \scriptscriptstyle |\beta'+\beta''+\gamma|\leq |f|\\ \scriptscriptstyle 
 |\nu|-1\leq |g'|,g''_j>0}} 
g''_u {\bf C}_f(\varphi^D,\beta'+\beta''+\gamma)\, 
 D^*_{\beta'} \left( {\bf C}_{g'}(\varphi,\nu-v^i) 
\right)\,
D^*_{\beta''} \left( {\bf C}_{g''}(\varphi,v^i)
\right) =
&
\\
& \displaystyle
\sum_{\substack{\scriptscriptstyle \rho+g'' = e\\ \scriptscriptstyle |\beta''+h|-1\leq |\rho|\\ \scriptscriptstyle 
g''_j>0}} g''_j
\left(
  \sum_{\substack{\scriptscriptstyle f + g' = \rho\\ 
          \scriptscriptstyle |\beta''+\gamma+\beta'|\leq |f|\\ 
           \scriptscriptstyle |\nu|-1\leq |g'|}} 
  {\bf C}_f(\varphi^D,\beta''+\gamma+\beta')\, 
  D^*_{\beta'} \left( {\bf C}_{g'}(\varphi,\nu-v^i) \right)
\right)
 D^*_{\beta''} \left( {\bf C}_{g''}(\varphi,v^i) \right) \stackrel{(\star)}{=}
&
\\
& \displaystyle
\sum_{\substack{\scriptscriptstyle \rho+g'' = e\\ \scriptscriptstyle |\beta''+h|-1\leq |\rho|\\ \scriptscriptstyle 
g''_j>0}} g''_j\,
{\bf C}_\rho(\varphi^D,\beta''+\gamma+\nu-v^i)\,
D^*_{\beta''} \left( {\bf C}_{g''}(\varphi,v^i) \right),
\end{eqnarray*}
where equality $(\star)$ comes from Proposition \ref{prop:varphi-D-main}, 
and so $N^{j,i}_{e,\nu,\gamma}$ only depends on $h=\gamma+\nu$. By setting:
$$
{\bf N}^{j,i}_{e,h} := 
\sum_{\substack{\scriptscriptstyle f + g = e\\ \scriptscriptstyle |\beta+h|-1\leq |f|\\ \scriptscriptstyle 
 g_j>0}} g_j
{\bf C}_f(\varphi^D,\beta+h-v^i)D^*_\beta({\bf C}_g(\varphi,v^i)) 
$$
for $e_j, h_i > 0, |e|\geq |h|$ and ${\bf N}^{j,i}_{e,h} := 0$ otherwise, we have
$N^{j,i}_{e,\nu,\gamma} = {\bf N}^{j,i}_{e,\gamma+\nu}$, and so:

\begin{eqnarray*}
& \displaystyle
\varepsilon^j(\varphi\sbullet r) = \cdots = \sum_{\scriptscriptstyle e_j>0} \left(
\sum_{\substack{\scriptscriptstyle   0<|h|\leq |e|\\ \scriptscriptstyle   o\in\supp h}}
\sum_{\substack{\scriptscriptstyle \gamma+\nu=h\\ \scriptscriptstyle i\in \supp \nu}}
N^{j,i}_{e,\nu,\gamma}
\nu_i r^*_\gamma  r_\nu  \right) 
\bfu^e =
&
\\
& \displaystyle
\sum_{\scriptscriptstyle e_j>0} \left(
\sum_{\substack{\scriptscriptstyle   0<|h|\leq |e|\\ \scriptscriptstyle   i\in\supp h}}
\sum_{\substack{\scriptscriptstyle \gamma+\nu=h\\ \scriptscriptstyle i\in \supp \nu}}
{\bf N}^{j,i}_{e,\gamma+\nu}
\nu_i r^*_\gamma  r_\nu  \right) 
\bfu^e =
\sum_{\scriptscriptstyle e_j>0} \left(
\sum_{\substack{\scriptscriptstyle   0<|h|\leq |e|\\ \scriptscriptstyle   i\in\supp h}}
{\bf N}^{j,i}_{e,h}
\sum_{\substack{\scriptscriptstyle \gamma+\nu=h\\ \scriptscriptstyle i\in \supp \nu}}
\nu_i r^*_\gamma  r_\nu  \right) 
\bfu^e =
&
\\
& \displaystyle
\sum_{\scriptscriptstyle e_j>0} \left(
\sum_{\substack{\scriptscriptstyle   0<|h|\leq |e|\\ \scriptscriptstyle   i\in\supp h}}
{\bf N}^{j,i}_{e,h}
\varepsilon^i_h(r)
\right) 
\bfu^e.
\end{eqnarray*}
\end{proof}

\begin{cor} Under the hypotheses of Theorem \ref{thm:epsilon-bullet}, we have:
$$
\varepsilon^j_e(\varphi\sbullet D) = \sum_{\substack{\scriptscriptstyle   0<|h|\leq |e|\\ \scriptscriptstyle i\in \supp h}}
{\bf N}^{j,i}_{e,h}
\varepsilon^i_h(D)\quad \forall e\in \Omega, \forall j=1,\dots,q. 
$$
\end{cor}

\bibliographystyle{amsplain}

\begin{thebibliography}{99}

\bibitem{has37}
 H.~Hasse and F.~K.~Schmidt.
 \newblock \href{https://doi.org/10.1515/crll.1937.177.215}{Noch eine
Begr\"{u}ndung der Theorie der h\"{o}heren Differrentialquotienten in einem algebraischen
Funktionenk\"orper einer Unbestimmten}.
\newblock {\em  J. Reine Angew. Math.} 177 (1937), 223-239.

\bibitem{haze_2011}
M.~Hazewinkel.
\newblock \href{https://doi.org/10.3390/axioms1020149}{Hasse--Schmidt Derivations and the Hopf Algebra of Non-Commutative Symmetric Functions}. {\em Axioms} 1 (2) (2012), 149--154.
\newblock ({\tt  \href{https://arxiv.org/abs/1110.6108}{arXiv:1110.6108}}).

\bibitem{heer70}
N.~Heerema.
\newblock \href{https://doi.org/10.1090/S0002-9904-1970-12609-X}{Higher derivations and automorphisms of complete local rings}.
\newblock {\em Bull. Amer. Math. Soc.} 76 (1970), 1212--1225.

\bibitem{hoff_kow_2015}
D.~Hoffmann and P.~Kowalski.
\newblock \href{https://doi.org/10.1016/j.jpaa.2014.05.024}{Integrating Hasse--Schmidt derivations}.
\newblock {\em J. Pure Appl. Algebra}, 219 (2015), 875--896.

\bibitem{mat-intder-I}
H.~Matsumura.
\newblock \href{https://projecteuclid.org/euclid.nmj/1118786907}{Integrable derivations}.
\newblock  {\em Nagoya Math. J.} 87 (1982), 227--245.

\bibitem{mat_86}
H.~Matsumura.
\newblock Commutative Ring Theory. Vol. 8 of Cambridge studies in
  advanced mathematics,
\newblock Cambridge Univ. Press, Cambridge, 1986.

\bibitem{Mirza_2010}  
M.~Mirzavaziri.
\newblock \href{http://dx.doi.org/10.1080/00927870902828751}{Characterization of higher derivations on algebras}.
\newblock {\em Comm. Algebra}, 38 (3) (2010), 981--987.

\bibitem{nar_2012}
L.~Narv\'aez Macarro.
\newblock \href{https://doi.org/10.1016/j.aim.2012.01.015}{On the modules of {$m$}-integrable derivations in non-zero characteristic}.
\newblock {\em Adv. Math.} 229 (5) (2012), 2712--2740.
\newblock (\href{https://arxiv.org/abs/1106.1391}{\tt arXiv:1106.1391}).

\bibitem{nar_subst_2018}
L.~Narv\'{a}ez~Macarro.
\newblock \href{https://doi.org/10.1007/978-3-319-96827-8_10}{On Hasse--Schmidt Derivations: the Action of Substitution Maps}.
\newblock In ``Singularities, Algebraic Geometry, Commutative Algebra, and Related Topics. Festschrift for Antonio Campillo on the occasion of his 65th birthday'', Greuel, Gert-Martin (ed.) et al., 219--262.
\newblock Springer International Publishing, Cham, 2018.
\newblock ({\tt  \href{https://arxiv.org/abs/1802.09894}{arXiv:1802.09894}}).
  
\bibitem{nar_envelop}
L.~Narv\'{a}ez~Macarro.  
\newblock \href{https://doi.org/10.1016/j.jpaa.2019.05.009}{Rings of differential operators as enveloping algebras of Hasse--Schmidt derivations}.
\newblock  {\em J. Pure Appl. Algebra} 224 (1) (2020), 320--361.   
\newblock ({\tt  \href{https://arxiv.org/abs/1807.10193}{arXiv:1807.10193}}).

\bibitem{nar_HSmod_vs_IC}
L.~Narv\'{a}ez~Macarro.  
\newblock Hasse--Schmidt modules versus integrable connections.
\newblock ({\tt  \href{https://arxiv.org/abs/1903.08985}{arXiv:1903.08985}}).

\bibitem{Saymeh_1986}
S.~A.~Saymeh.
\newblock \href{http://projecteuclid.org/euclid.ojm/1200779341}{On Hasse--Schmidt higher derivations}.
\newblock {\em Osaka J. Math.} 23 (2) (1986), 503--508.

\bibitem{MPTirado_2018}
M.~P.~Tirado~Hern\'andez.
\newblock Integrable derivations in the sense of Hasse--Schmidt for some binomial plane curves. Rend. Sem. Mat. Univ. Padova, to appear.
\newblock ({\tt  \href{https://arxiv.org/abs/1807.10502}{arXiv:1807.10502}}).


\end{thebibliography}

\end{document}